\documentclass[psamsfonts]{amsart}

\usepackage[a4paper,bindingoffset=0.0in,left=.975in,right=.975in,top=1.50in,bottom=1.50in,footskip=.35in]{geometry}

\usepackage{amssymb,amsfonts}
\usepackage[all,arc]{xy}
\usepackage{enumerate}
\usepackage{mathrsfs}

\usepackage{setspace}

\newtheorem{thm}{Theorem}[section]
\newtheorem{cor}[thm]{Corollary}
\newtheorem{prop}[thm]{Proposition}
\newtheorem{lem}[thm]{Lemma}

\theoremstyle{definition}
\newtheorem{defn}[thm]{Definition}

\theoremstyle{remark}
\newtheorem{rem}[thm]{Remark}

\newcommand{\Q}{\Bbb{Q}}

\newcommand{\OO}{\mathcal{O}}

\newcommand{\HH}{\mathcal{H}}

\newcommand{\F}{\Bbb{F}}

\newcommand{\Z}{\Bbb{Z}}

\newcommand{\N}{\Bbb{N}}

\newcommand{\m}{\frak{m}}

\newcommand{\Mod}{\text{Mod}}

\newcommand{\ind}{\text{ind}}

\newcommand{\gr}{\text{gr}}

\newcommand{\Hom}{\text{Hom}}

\newcommand{\Ob}{\text{Ob}}

\newcommand{\rk}{\text{rk}}

\newcommand{\Ext}{\text{Ext}}

\newcommand{\End}{\text{End}}

\newcommand{\Rep}{\text{Rep}}

\newcommand{\Spec}{\text{Spec}}

\newcommand{\SL}{\text{SL}}

\newcommand{\op}{\text{op}}

\newcommand{\y}{\hspace{6pt}}

\makeatletter
\let\c@equation\c@thm
\makeatother
\numberwithin{equation}{section}

\title{Koszul duality for Iwasawa algebras modulo $p$}

\author{Claus Sorensen}

\date{}

\begin{document}

\begin{abstract}
In this article we establish a version of Koszul duality for filtered rings arising from $p$-adic Lie groups. Our precise setup is the following. We let $G$ be a uniform pro-$p$ group
and consider its completed group algebra $\Omega=k[\![G]\!]$ with coefficients in a finite field $k$ of characteristic $p$. It is known that $\Omega$ carries a natural filtration and
$\gr \Omega=S(\frak{g})$ where $\frak{g}$ is the (abelian) Lie algebra of $G$ over $k$. One of our main results in this paper is that the Koszul dual $\gr \Omega^!=\bigwedge \frak{g}^{\vee}$ can be promoted to an $A_{\infty}$-algebra in such a way that the derived category of pseudocompact $\Omega$-modules $D(\Omega)$ becomes equivalent to the 
derived category of strictly unital $A_{\infty}$-modules $D_{\infty}(\bigwedge \frak{g}^{\vee})$. In the case where $G$ is an abelian group we prove that the $A_{\infty}$-structure is trivial and deduce an equivalence between  $D(\Omega)$ and the derived category of differential graded modules over $\bigwedge \frak{g}^{\vee}$ which generalizes a result of Schneider for $\Z_p$.
\end{abstract}

\maketitle

\tableofcontents


\onehalfspacing

\section{Introduction}

Classical Koszul duality is a derived equivalence between modules over the symmetric algebra $S(V)$ and modules over the exterior algebra $\bigwedge V^{\vee}$.  
Here $V$ is a finite-dimensional vector space over a field $k$, and $V^{\vee}=\Hom_k(V,k)$ is its dual space. More precisely one considers the category of finitely generated graded left modules over $S(V)$ and its bounded derived category $D^b(S(V))$. Similarly for the exterior algebra. Then Koszul duality gives an equivalence of triangulated categories
$$
D^b(S(V)) \overset{\sim}{\longrightarrow}  D^b(\bigwedge V^{\vee}).
$$
This is \cite[Thm.~3]{BGG78}. The equivalence of course derives its name from the Koszul resolution
$$
\cdots \longrightarrow S(V) \otimes \bigwedge^2 V \longrightarrow S(V) \otimes \bigwedge^1 V \longrightarrow S(V)\longrightarrow k \longrightarrow 0.
$$
Building on work of Priddy, L\"{o}fwall, Backelin, and others, this was generalized to more general rings in \cite{BGS96}. They define a Koszul algebra to be a graded $k$-algebra 
$A=k \oplus A_1\oplus A_2 \oplus \cdots$ for which $k$ admits a (graded) projective resolution $P_{\bullet}\rightarrow k$ such that each $P_i$ is generated by its degree $i$ component ($P_i=AP_{i,i}$). The Koszul dual algebra $A^!$ is defined by viewing $A$ as a quadratic algebra and then dualizing the relations, cf. \cite[Def.~2.8.1]{BGS96}. Alternatively one can identify it with the (opposite) Yoneda algebra $A^!=\Ext_A^*(k,k)^{\op}$ of self-extensions of $k$, cf. \cite[Thm.~2.10.1]{BGS96}.
Assuming $A$ is locally finite-dimensional, that is $\dim_k(A_i)<\infty$ for all $i$, the Koszul dual $A^!$ is again Koszul and canonically $A^{!!} \overset{\sim}{\longrightarrow}A$.
In the classical case $A=S(V)$ has Koszul dual $A^!=\bigwedge V^{\vee}\simeq \Ext_{S(V)}^*(k,k)^{\op}$. Still under the assumption that $A$ is {\it{locally}} finite-dimensional,
\cite[Thm.~2.12.1]{BGS96} establishes an equivalence
$$
D^{\downarrow}(A) \overset{\sim}{\longrightarrow} D^{\uparrow}(A^!)
$$
where $D^{\downarrow}$ and $D^{\uparrow}$ are variants of the derived categories. For instance $D^{\downarrow}(A)$ is obtained by inverting quasi-isomorphisms in the category 
$C^{\downarrow}(A)$ of complexes of graded $A$-modules $\cdots \rightarrow M^i \rightarrow M^{i+1}\rightarrow \cdots$ satisfying $M_j^i=0$ if $i<\!\!<0$ or $i+j>\!\!>0$
(no other finiteness assumptions on the $M^i$). 
Assuming $A$ itself is finite-dimensional (hence $A_i=0$ for $i>\!\!>0$), and that $A^!$ is Noetherian, \cite[Thm.~2.12.6]{BGS96} gives an equivalence between bounded derived categories
$$
D^b(A) \overset{\sim}{\longrightarrow} D^b(A^!).
$$
Here, as in the classical case, we are deriving the categories of finitely generated graded modules.

There have been several attempts to extend Koszul duality to filtered algebras, notably the works \cite{Pos93, Flo06} of Fl\o{}ystad and Positselski which share some important features with the main results of this paper. Given a Koszul algebra $A$ as above, Positselski considers filtered rings $U$ with $\gr(U)\simeq A$ and shows that this amounts to 
promoting $A^!$ to a "curved" differential graded algebra (DGA). When $U$ is furthermore augmented $A^!$ is an actual DGA. (This is our main case of interest so we will not recall here what it means to be curved.) Fl\o{}ystad took this further and constructed a pair of adjoint functors between homotopy categories $K(U)\rightleftarrows K(A^!,d)$. He showed they descend to an equivalence $\tilde{D}(U)\overset{\sim}{\longrightarrow} \tilde{D}(A^!,d)$ but here $\tilde{D}$ is {\it{not}} the derived category, it is rather squeezed in between the homotopy category and the derived category. We should emphasize that the filtered rings $U$ considered in \cite{Pos93} are of a completely different nature than the Iwasawa algebras we will consider in this paper. Positselski requires $U$ to be a non-homogeneous quadratic algebra $U=T(V)/(P)$ for some vector space $V$ and ideal of relations $(P)$ in the tensor algebra generated by a subspace $P \subset k \oplus V \oplus V^{\otimes 2}$ with $P \cap (k\oplus V)=0$. He associates the quadratic algebra $A=T(V)/(Q)$ where
$Q \subset V^{\otimes 2}$ is the image of $P$ under projection and decrees that the natural map $A \rightarrow \gr U$ is an isomorphism. 

Analogously we will consider completed group rings $\Omega=\Omega(G)$ modulo $p$ of certain compact $p$-adic Lie groups $G$ for which $\gr \Omega\simeq S(\frak{g})$
where $\frak{g}=\text{Lie}(G)$ is the Lie algebra, and promote the Koszul dual algebra $\gr\Omega^!\simeq \bigwedge \frak{g}^{\vee}$ to a so-called $A_{\infty}$-algebra in such a way that $D(\Omega)\overset{\sim}{\longrightarrow} D_{\infty}(\bigwedge\frak{g}^{\vee})$. The notation and terminology will be explained in more detail below.

Let us state our main results precisely. Let $G$ be a $p$-adic Lie group which is torsionfree\footnote{Positselski has kindly informed us that Theorem \ref{one} can be extended to {\it{any}} pro-$p$ group if one replaces $D(\Omega)^{\op}$ by the so-called {\it{co-derived}} category of smooth $G$-modules, by applying the results from \cite[Sect.~6.4]{Pos11}.} and pro-$p$. Let $\Omega=k[\![G]\!]$ be its completed group ring over a finite field $k$ of characteristic $p$, and consider the derived category $D(\Omega)$ of the category of pseudocompact left $\Omega$-modules. 

\begin{thm}\label{one}
The (opposite) Yoneda algebra $\Omega^!:=\Ext_{\Omega}^*(k,k)^{\op}$ has an $A_{\infty}$-algebra structure, defined uniquely up to $A_{\infty}$-isomorphism, for which there is an equivalence of triangulated categories
$$
D(\Omega)\overset{\sim}{\longrightarrow} D_{\infty}(\Omega^!)
$$
where the target is the derived category of strictly unital left $A_{\infty}$-modules over $\Omega^!$. 
\end{thm}

When $G$ is a uniform pro-$p$ group (see the discussion after corollary \ref{cohom}) one can compute $\Omega^!$ using Lazard's calculation of the mod $p$ cohomology ring of an equi-$p$-valued group and deduce the next result.

\begin{thm}\label{two}
Let $G$ be a uniform pro-$p$ group with Lie algebra $\frak{g}$ over $k$ (which is necesarily abelian). Then 
$\bigwedge \frak{g}^{\vee}$ has an $A_{\infty}$-algebra structure, defined uniquely up to $A_{\infty}$-isomorphism, for which there is an equivalence of triangulated categories
$$
D(\Omega)\overset{\sim}{\longrightarrow} D_{\infty}(\bigwedge \frak{g}^{\vee}).
$$
Furthermore, when $G$ is abelian the $A_{\infty}$-algebra structure on $\bigwedge \frak{g}^{\vee}$ is trivial (meaning all higher multiplication maps $\mu_n$ vanish for $n>2$) and the target is the derived category of differential graded modules.
\end{thm}

For example this theorem applies to the congruence subgroups $G=\ker\big(\mathcal{G}(\Z_p)\rightarrow \mathcal{G}(\Z/p^m\Z) \big)$ of any finite type smooth affine group scheme 
$\mathcal{G}_{/\Z_p}$-- at least for $m\gg 0$. Let us also point out that any $p$-adic Lie group contains a uniform open subgroup, cf. \cite[p.~167]{Laz65}.

The case $G=\Z_p$ of theorem \ref{two} is a result of Schneider. In \cite[p.~460]{Sch15} it is shown that there is an equivalence 
$D(\Z_p)\overset{\sim}{\longrightarrow} D(k[\varepsilon])$ where $k[\varepsilon]=k\oplus k\varepsilon$ is the algebra of dual numbers ($\varepsilon^2=0$) thought of as a differential graded algebra concentrated in degrees $0$ and $1$ with zero differential. This is a special case of \ref{two} by observing that $k[\varepsilon]\simeq \bigwedge k$, and a computation in Hochschild cohomology shows $\bigwedge k$ is intrinsically formal -- it has no non-trivial minimal $A_{\infty}$-structures whatsoever, cf. remark \ref{d=1} for more details.

Schneider's paper \cite{Sch15} in fact plays a pivotal role in the proof of our two theorems, which we will now outline. The equivalence in theorem \ref{one} factors as a composition
(after taking opposites)
$$
D(\Omega)^{\op} \overset{(1)}{=\joinrel=\joinrel=}
D(G) \overset{(2)}{=\joinrel=\joinrel=}
D(\HH^{\bullet}) \overset{(3)}{=\joinrel=\joinrel=}
D_{\infty}(h^*(\HH^{\bullet})) \overset{(4)}{=\joinrel=\joinrel=}
D_{\infty}(\Ext_{\Omega}^*(k,k)).
$$ 
Here we use $=\joinrel=\joinrel=$ instead of $\overset{\sim}{\longrightarrow}$ for clarity. We explain each step (1)--(4) separately.

\medskip

\begin{itemize}
\item[(1)] (Section \ref{dual}) The first step is just Pontryagin duality between the category $\Mod(\Omega)$ of pseudocompact $\Omega$-modules and the 
smooth $k[G]$-modules $\Rep_k^{\infty}(G)$. Taking continuous $k$-linear duals defines mutually quasi-inverse contravariant functors which are exact, and therefore induce an anti-equivalence on the level of derived categories.

\medskip

\item[(2)] (Section \ref{lie}) This is the main result of \cite{Sch15} in the special case $G=I$. We discuss this case in some detail in the main body of the text, but claim no novelty here (except for the exposition). The result is based on a theorem of Keller which says that any algebraic triangulated category $\mathcal{D}$ with (coproducts and) a compact generator $C$ is equivalent to $D(\HH_C^{\bullet})$ for some naturally defined differential graded algebra $\HH_C^{\bullet}$, see section \ref{morit}. Schneider shows that Keller's theorem applies to $\mathcal{D}=D(G)$ for a torsionfree pro-$p$ group $G$, and $C=k$ the trivial representation. The crucial step is to verify that $k$ is a generator of $D(G)$, cf. proposition \ref{gen}
below. For that we need $G$ to be torsionfree to guarantee it has finite cohomological $p$-dimension. The DGA is here simply $\HH^{\bullet}=\End_G^{\bullet}(I^{\bullet})^{\op}$ 
for a choice of injective resolution $k \rightarrow I^{\bullet}$. Alternatively $\HH^{\bullet}=\End_{\Omega}^{\bullet}(P^{\bullet})$ for a choice of projective resolution 
$P^{\bullet}\rightarrow k$, by Pontryagin duality.

\medskip

\item[(3)] (Section \ref{model}+\ref{dercat}) By a general result of Kadeishvili's from the 80s the cohomology algebra $h^*(A^{\bullet})$ of any DGA $A^{\bullet}$ (even of any $A_{\infty}$-algebra) has an $A_{\infty}$-structure $(\mu_n)_{n \geq 1}$ with $\mu_1=0$ ("minimality") and $\mu_2$ induced by multiplication on $A^{\bullet}$, for which there is a quasi-isomorphism
$f: h^*(A^{\bullet})\rightarrow A^{\bullet}$. We will discuss this in detail in section \ref{model}; here we will just note that each $\mu_n$ is a higher multiplication map
$A^{\bullet \otimes n}\rightarrow A^{\bullet}$ of degree $2-n$, and the sequence $(\mu_n)_{n \geq 1}$ satisfies generalized associativity laws known as the Stasheff identities. One can restrict $A_{\infty}$-modules along $f$ and this gives rise to an equivalence $D(A^{\bullet})\overset{\sim}{\longrightarrow} D_{\infty}(h^*(A^{\bullet}))$ as discussed in section \ref{dercat}. The latter is due to Lef\`{e}vre-Hasegawa. Applying this to $\HH^{\bullet}$ yields (3). 

\medskip

\item[(4)] (Section \ref{yon}) This is merely the computation $h^*(\HH^{\bullet})=h^*(\End_{\Omega}^{\bullet}(P^{\bullet}))=\Ext_{\Omega}^*(k,k)$.

\end{itemize}

\medskip

\noindent The first half of theorem \ref{two} is a corollary of theorem \ref{one} by calculating $\Ext_{\Omega}^*(k,k)$ for uniform pro-$p$ groups (equivalently $G$ which admit a $p$-valuation $\omega$ such that $(G,\omega)$ is equi-$p$-valued, see section \ref{valu}). Namely
$$
\Omega^!=\Ext_{\Omega}^*(k,k)^{\op}=\Ext_G^*(k,k)=H^*(G,k)=\joinrel=\joinrel=\bigwedge \frak{g}^{\vee}
$$ 
where the last step is due to Lazard, cf. corollary \ref{cohom} below. The second half (where $G$ is abelian) is proved in section \ref{prooftwo}. The basic idea is to think of $\Omega$ as a completed symmetric algebra $\Omega=\widehat{S(V)}$ and take $P^{\bullet}\rightarrow k$ to be the completion of the Koszul resolution $K^{\bullet}$. Since $S(V)$ is Koszul $\Ext_{S(V)}^*(k,k)$ is formal, from which the result follows after observing that there is a quasi-isomorphism $\End_{S(V)}^{\bullet}(K^{\bullet})\rightarrow \HH^{\bullet}$. 

In fact theorem \ref{one} has a natural extension to arbitrary $p$-adic Lie groups $G$ (not necessarily compact). As in \cite{Sch15} fix a torsionfree pro-$p$ subgroup
$I \subset G$. Schneider shows that $\ind_I^G(1)$ is a compact generator of $D(G)$ and deduces an equivalence with $D(\HH^{\bullet})$ from Keller's theorem. Here 
the Hecke DGA is defined as $\HH^{\bullet}=\End_G^{\bullet}(I^{\bullet})^{\op}$ for a choice of injective resolution $\ind_I^G(1)\longrightarrow I^{\bullet}$. The Iwasawa algebra is defined to be
$\Omega=k[G] \otimes_{k[I]} k[\![I]\!]$, and Pontryagin duality extends to an equivalence $D(\Omega)^{\op}=D(G)$ as explained in \cite[Thm.~1.5]{Koh17}. Altogether
the mod $p$ "derived Hecke algebra" (cf. \cite{Ven17})
$$
\Omega^!:=\Ext_G^*(\ind_I^G(1),\ind_I^G(1))\simeq H^*(I,\ind_I^G(1))
$$
admits an $A_{\infty}$-structure such that $D(\Omega) \overset{\sim}{\longrightarrow} D_{\infty}(\Omega^!)$. We have chosen not to emphasize this stronger result since the computation of $H^*(I,\ind_I^G(1))$ appears to be quite delicate in general. The recent paper \cite{OS18} deals with the case $\SL_2(F)$ for arbitrary $F/\Q_p$, which is already quite involved. See \cite[Thm.~3.23]{OS18} for sample calculations.

We should point out that Theorems \ref{one} and \ref{two} were known to some experts in the field, although they never appeared explicitly in print. Also, the $A_{\infty}$-structures we work with are only uniquely defined up a {\it{non}}-canonical $A_{\infty}$-isomorphism -- they depend on several choices; see the end of Section \ref{model} for instance. This severely limits the usage of our results beyond the case of abelian groups $G$. However, we feel that Theorems \ref{one} and \ref{two} are of some interest in their own right -- at least from a philosophical standpoint.

In order to make this article more accessible to mathematicians from different disciplines, we have included lots of details and background on the Keller-Schneider equivalence, 
Lazard's theory of $p$-valued groups and their cohomology, $A_{\infty}$-algebras and modules, and other material relevant for the paper.  

\section{Notation}\label{not}

Throughout $p$ denotes a fixed odd prime. Similar arguments most likely go through for $p=2$ with minor modifications, but we have not checked this in detail. One issue is the notion of a uniform pro-$p$ group $G$ in the case $p=2$ where one should require that $G/\overline{G^4}$ is abelian.

If $\mathcal{A}$ is an abelian category its opposite category $\mathcal{A}^{\op}$ is again an abelian category; kernels in $\mathcal{A}$ correspond to cokernels in $\mathcal{A}^{\op}$ and vice versa (similarly for images and coimages), cf. \cite[Lem.~12.5.2]{Sta18}. If $\mathcal{C}$ is any category, when we view $X \in \Ob(\mathcal{C})$ as an object of $\mathcal{C}^{\op}$ we will write $X^{\op}$ to avoid confusion; similarly for morphisms. 

A triangulated category is an additive category $\mathcal{D}$ with a translation functor $T: \mathcal{D} \rightarrow \mathcal{D}$ (an auto-equivalence) and a class of distinguished triangles $X\rightarrow Y \rightarrow Z \rightarrow TX$ satisfying four axioms TR1--4, cf. \cite[Def.~13.3.2]{Sta18}.
The opposite category $\mathcal{D}^{\op}$ is again a triangulated category. Its translation functor is the inverse $T^{-1}$ and the triangle $X^{\op}\rightarrow Y^{\op}\rightarrow Z^{\op}\rightarrow T^{-1}X^{\op}$ is declared to be distinguished precisely when
$T^{-1}X \rightarrow Z \rightarrow Y \rightarrow X$ is distinguished in $\mathcal{D}$, cf. \cite[Rem.~1.1.5]{Nee01}.

If $\mathcal{A}$ is an abelian category $C(\mathcal{A})$ denotes the category of complexes $\cdots\rightarrow X^i \rightarrow X^{i+1}\rightarrow \cdots$. We let $K(\mathcal{A})$ be the homotopy category and $D(\mathcal{A})$ is the (unbounded) derived category; both are triangulated categories in the usual fashion, cf. \cite[Ch.~I]{Har66}. Taking opposites is compatible with derived categories, to the effect that there is a natural equivalence of triangulated categories $D(\mathcal{A}^{\op})\overset{\sim}{\longrightarrow} D(\mathcal{A})^{\op}$.

\section{Pontryagin duality for profinite groups}\label{dual}

In this section $G$ denotes an arbitrary profinite group and $k$ is any field. We let $\Omega=\Omega(G)=k[\![G[\!]$ be the Iwasawa $k$-algebra of $G$. We will usually suppress $G$ from the notation and just write $\Omega$. It has a natural augmentation map $\Omega \overset{a}{\longrightarrow} k$ and carries the inverse limit topology. We always endow $k$ with the discrete topology. Thus $\Omega=\varprojlim k[G/N]$ (where $N \subset G$ runs over the open normal subgroups) is a {\it{pseudocompact}} ring; meaning it is a complete Hausdorff topological ring which has a fundamental system of open neighborhoods of zero consisting of two-sided ideals $I_N$ such that $\Omega/I_N$ is Artinian. 

In this paper we will mostly be interested in the category $\Mod(\Omega)$ of pseudocompact left $\Omega$-modules (and continuous homomorphisms). Its objects can be defined 
as pseudocompact $k$-vector spaces $M$ (that is inverse limits of finite-dimensional $k$-vector spaces) with a $k$-linear jointly continuous action $G \times M \longrightarrow M$. One can immediately verify that the $k[G]$-module structure of $M$ then extends uniquely to a continuous map $\Omega\times M \longrightarrow M$, and that $M$ admits a neighborhood basis at zero consisting of $\Omega$-submodules.  

It is well-known that $\Mod(\Omega)$ is an abelian category, which admits arbitrary products, and the formation of filtered (inverse) limits is exact, cf. \cite[p.~392]{Gab62}.  
Moreover $\Mod(\Omega)$ has enough projectives, cf. \cite[Lem.~1.6]{Bru66}.

Dually we have the abelian category $\Rep_k^{\infty}(G)$ of {\it{smooth}} $G$-representations on $k$-vector spaces $V$ (and $k[G]$-linear maps). In other words the action
$G \times V \longrightarrow V$ is continuous for the discrete topology on $V$. The Pontryagin dual of $V$ is defined as $V^{\vee}=\Hom_k(V,k)$ equipped with the topology of pointwise convergence (i.e., induced from the product topology on $k^V$). With the contragredient $G$-action $V^{\vee}$ thus becomes an object of $\Mod(\Omega)$. Conversely, starting from a pseudocompact $\Omega$-module $M$ we define its dual to be $M^{\vee}=\Hom_k^{\text{cts}}(M,k)$ with the discrete topology. It is well-known that this sets up a duality between the two categories. See \cite{Sch95, ST02, Eme10} and the more recent \cite{Koh17}. We summarize this below. 

\begin{thm}\label{pont}
The Pontryagin duality functors $(\cdot)^{\vee}$ are exact and mutually quasi-inverse. Thus
$$
(\cdot)^{\vee}: \Mod(\Omega)^{\op} \longrightarrow \Rep_k^{\infty}(G)
$$
is an equivalence of categories. Furthermore $\Omega^{\vee}=\mathcal{C}^{\infty}(G,k)$.  
\end{thm}

\begin{proof}
Apart from exactness this is essentially \cite[Thm.~1.5]{Koh17} in the case of a compact group.
\end{proof}

In particular we infer that $\Rep_k^{\infty}(G)$ admits arbitrary coproducts (direct sums), that filtered colimits are exact, and it has enough injectives. In fact 
$\Rep_k^{\infty}(G)$ is a {\it{Grothendieck}} category since it has a generator $Y=\oplus_N \ind_N^G(1)$. Indeed, Frobenius reciprocity 
$\Hom_G(Y,V)=\prod_N V^N$ shows the functor $\Hom_G(Y,-)$ is faithful. 

Because $(\cdot)^{\vee}$ is exact it preserves quasi-isomorphisms and induces an equivalence on the level of derived categories. Throughout the paper we let $D(G)=D(\Rep_k^{\infty}(G))$ and $D(\Omega)=D(\Mod(\Omega))$. Similarly for the category of complexes $C(G)$ and the homotopy category $K(G)$. The same for $\Omega$.

\begin{cor}\label{pd}
The functors $(\cdot)^{\vee}$ define an equivalence of triangulated categories $D(\Omega)^{\op}\overset{\sim}{\longrightarrow} D(G)$. 
\end{cor}

Here the opposite triangulated category $D(\Omega)^{\op}$ is equivalent to $D(\Mod(\Omega)^{\op})$ as explained in the notation section \ref{not}.

\section{Keller's Morita theorem for DGA's}\label{morit}

If $\mathcal{A}$ is an abelian category with coproducts which admits a compact (i.e., $\Hom_{\mathcal{A}}(P,-)$ commutes with coproducts) projective generator $P$, then a theorem of  Morita states that $\Hom_{\mathcal{A}}(P,-)$ gives an equivalence $\mathcal{A} \longrightarrow \Mod\big(\End_{\mathcal{A}}(P)^{\op}\big)$ onto the category of {\it{left}} modules over 
the algebra $\End_{\mathcal{A}}(P)^{\op}$. 

In \cite{Kel94} Keller proved an analogue of Morita's theorem for differential graded algebras (from now on referred to as DGA's). Here we will follow \cite{Kel07} in exposing his result. 
First recall that a DGA over a field $k$ is a $\Z$-graded associative $k$-algebra $\HH^{\bullet}=\bigoplus_{i \in \Z} \HH^i$ with a degree one $k$-linear operator 
$d: \HH^{\bullet} \rightarrow \HH^{\bullet}$ such that $d^2=0$ which satisfies the Leibniz rule
$$
d(ab)=d(a)b+(-1)^{\deg a}ad(b)
$$
for all $a, b \in \HH^{\bullet}$ with $a$ homogeneous. A DG left $\HH^{\bullet}$-module is a $\Z$-graded $\HH^{\bullet}$-module $M^{\bullet}=\bigoplus_{i \in \Z}M^i$ with a $k$-linear differential $d: M^{\bullet}\rightarrow M^{\bullet}$ of degree one which satisfies an analogous Leibniz rule. In this situation the cohomology algebra
$h^*(\HH^{\bullet})=\bigoplus_{i \in \Z} h^i(\HH^{\bullet})$ is a graded algebra, and $h^*(M^{\bullet})=\bigoplus_{i \in \Z} h^i(M^{\bullet})$ is a graded (left) module over $h^*(\HH^{\bullet})$. 

We follow the conventions of \cite[p.~68]{BL94} regarding DGA's and DG-modules all being unital. More precisely we always assume the existence of a multiplicative identity 
$1\in \HH^0$ such that $d(1)=0$. Then $h^*(\HH^{\bullet})$ inherits a multiplicative identity $1\in h^0(\HH^{\bullet})$. Moreover we assume $1 \in \HH^0$ acts trivially on all our DG-modules $M^{\bullet}$.

\begin{rem}\label{mor}
Our main examples of DG algebras and modules will arise as follows. Let $\mathcal{A}$ be an abelian category and consider two (unbounded) complexes 
$I^{\bullet}$ and $J^{\bullet}$ in $C(\mathcal{A})$. The associated {\it{morphism}} complex $\Hom_{\mathcal{A}}^{\bullet}(I^{\bullet},J^{\bullet})$ of abelian groups has $i$th component
(cf. \cite[p.~63]{Har66})
$$
\Hom_{\mathcal{A}}^{i}(I^{\bullet},J^{\bullet})=\prod_{q \in \Z} \Hom_{\mathcal{A}}(I^q,J^{q+i}).
$$
An element hereof is just a collection of morphisms $a=(a_q)_{q\in \Z}$ where each $a_q$ shifts the degrees by $i$ (there is no compatibility requirement with the differentials). 
The differentials are defined below. 
$$
\Hom_{\mathcal{A}}^{i}(I^{\bullet},J^{\bullet})\overset{d}{\longrightarrow} \Hom_{\mathcal{A}}^{i+1}(I^{\bullet},J^{\bullet}) \y \y \y
d(a)_q=d \circ a_q - (-1)^ia_{q+1}\circ d.
$$
If we take the {\it{same}} complex $J^{\bullet}=I^{\bullet}$ the morphism complex $\Hom_{\mathcal{A}}^{\bullet}(I^{\bullet},I^{\bullet})$ carries the structure of a DGA over $\Z$. 
Two homogeneous elements $a$ and $b$, of degree $i$ and $j$ respectively, are multiplied by the rule
$$
(ab)_q=(-1)^{ij}a_{q+j}\circ b_q.
$$
When $J^{\bullet}$ is arbitrary a similar formula endows $\Hom_{\mathcal{A}}^{\bullet}(I^{\bullet},J^{\bullet})$ with the structure of a DG right module over $\End_{\mathcal{A}}^{\bullet}(I^{\bullet})$; or what amounts to the same -- a DG left module over the opposite $\End_{\mathcal{A}}^{\bullet}(I^{\bullet})^{\op}$.

It follows straight from the definition of the differentials in the morphism complex that its $i$-cycles correspond to morphisms of complexes 
$I^{\bullet}\rightarrow J^{\bullet+i}$ (recall that the translation functor on $D(\mathcal{A})$ also changes signs of the differentials) and its $i$-boundaries are null-homotopic, cf. \cite[p.~64]{Har66}. Consequently
$$
h^i\big(\Hom_{\mathcal{A}}^{\bullet}(I^{\bullet},J^{\bullet})\big)=\Hom_{K(\mathcal{A})}(I^{\bullet},J^{\bullet+i}).
$$
When $\mathcal{A}$ has enough injectives and $J^{\bullet}$ is bounded below one can replace $K(\mathcal{A})$ by the derived category above, and identify 
the right-hand side with the $i$th hyperext $\Ext_{\mathcal{A}}^i(I^{\bullet},J^{\bullet})$, cf. \cite[Thm.~6.4]{Har66}.
\end{rem}

We return to a general DGA $\HH^{\bullet}$ over a field $k$, and the category of DG left $\HH^{\bullet}$-modules. There is an obvious notion of quasi-isomorphism, and by localization one obtains the derived category $D(\HH^{\bullet})$ which naturally becomes a triangulated category. As for abelian categories, morphisms in the derived category are easier to work with (as equivalence classes of roofs) if one localizes the homotopy category. In the DGA setup the morphisms of $K(\HH^{\bullet})$ are homotopy classes of maps between DG modules, where
$f: M^{\bullet}\rightarrow N^{\bullet}$ is null-homotopic if there is a map $r: M^{\bullet}\rightarrow N^{\bullet-1}$ of graded $\HH^{\bullet}$-modules for which $f=d \circ r +r \circ d$.
A streamlined exposition of these constructions is in \cite[Ch.~10]{BL94}.

Keller proved that the triangulated categories equivalent to some $D(\HH^{\bullet})$ can be characterized by a short list of properties. First, $D(\HH^{\bullet})$ is always {\it{algebraic}}.
Recall that a triangulated category $\mathcal{D}$ is said to be algebraic if it is equivalent to the stable category of a Frobenius category. We will not recall these notions here but just refer to \cite[Ch.~7]{Kra07} where this definition is discussed in great detail. What we will use later is the following simple criterion of Krause, see \cite[Lem.~7.5]{Kra07}.

\begin{lem}\label{kra}
A triangulated category $\mathcal{D}$ is algebraic if and only if there is an additive category $\mathcal{A}$ and a fully faithful exact functor 
$\mathcal{D}\longrightarrow K(\mathcal{A})$.
\end{lem}

Secondly, $D(\HH^{\bullet})$ admits (set-indexed) coproducts which can be computed as direct sums of DG $\HH^{\bullet}$-modules, and the module $\HH^{\bullet}$ is a {\it{compact generator}}. Recall that, when $\mathcal{D}$ admits coproducts, an object $C\in \Ob(\mathcal{D})$ is said to be compact if the functor $\Hom_{\mathcal{D}}(C,-)$ commutes with coproducts. We say that $C$ generates $\mathcal{D}$ if $X=0$ precisely when $\Hom_{\mathcal{D}}(C,T^iX)=0$ for all $i \in \Z$.

Keller's Morita theorem shows that these properties characterize the categories $D(\HH^{\bullet})$. 

\begin{thm}\label{kel}
Let $\mathcal{D}$ be an algebraic triangulated category which admits arbitrary coproducts and a compact generator $C\in \Ob(\mathcal{D})$. Then there is a differential graded algebra
$\HH_C^{\bullet}=R\Hom_{\mathcal{D}}(C,C)^{\op}$ and an equivalence of triangulated categories $\mathcal{D} \longrightarrow D(\HH_C^{\bullet})$.
\end{thm}

\begin{proof}
This is part of \cite[Thm.~8.7(b)]{Kel07} (using a different notation).
\end{proof}

In the theorem $R\Hom_{\mathcal{D}}(C,C)$ is a certain complex with $i$th cohomology $\Hom_{\mathcal{D}}(C,T^iC)$. We will only apply the theorem to situations where $\mathcal{D}=D(\mathcal{A})$ for some abelian category $\mathcal{A}$ with enough injectives and the compact generator is an object $C$ of  $\mathcal{A}$ thought of as the complex $C[0]$ concentrated in degree zero. Upon choosing an injective resolution $C[0]\rightarrow I^{\bullet}$  the associated DGA is $\mathcal{H}^{\bullet}=\End_{\mathcal{A}}^{\bullet}(I^{\bullet})^{\op}$ as defined in Remark \ref{mor}. In this situation the functor 
$D(\mathcal{A})\rightarrow D(\mathcal{H}^{\bullet})$ from Theorem \ref{kel} is the formation of the morphism complex $\Hom_{\mathcal{A}}^{\bullet}(I^{\bullet},-)$ as above -- after taking a $K$-injective resolution.

The cohomology algebra
$h^*(\mathcal{H}^{\bullet})$ has the following nice description in terms of the Yoneda algebra.
$$
h^*(\mathcal{H}^{\bullet})=\Ext_{\mathcal{A}}^*(C,C)^{\op}=\bigoplus_{i \geq 0} \Ext_{\mathcal{A}}^i(C,C).
$$
Similarly for the cohomology of $\Hom_{\mathcal{A}}^{\bullet}(I^{\bullet},-)$, cf. \cite[Cor.~6.5]{Har66} which compares $\Ext^i$ to hyperext.

\section{Schneider's equivalence for $p$-adic Lie groups}\label{lie}

In this section we assume $k$ is a field of characteristic $p>0$ and that $G$ is a torsionfree pro-$p$ group. Moreover, we will assume $G$ is a $p$-adic Lie group
of dimension $d=\dim(G)$ over $\Q_p$ say. 

By a result of Serre such $G$ have {\it{finite}} cohomological $p$-dimension $\text{cd}_p(G)=d$, cf. \cite[Cor.~(1)]{Ser65}. In that paper Serre also mentions that the 
$p$-dimension $\text{cd}_p$ is infinite if there are elements of order $p$. We will need this finiteness result to guarantee certain derived functors are defined on all of
$D(G)$.

Furthermore, our assumptions on $G$ ensure that $\Omega=k[\![G]\!]$ is a Noetherian ring which is local with maximal ideal $\m=\ker(a)$ since $G$ is pro-$p$. In particular the topology on $\Omega$, which is defined by the ideals $I_N=\ker\big(\Omega\rightarrow k[G/N]\big)$, coincides with the $\m$-adic topology. See \cite[Prop.~19.7, Lem.~19.8, Thm.~33.4]{Sch11} for these standard results about $p$-adic Lie groups.

Schneider's idea in \cite{Sch15} was to apply Keller's theorem \ref{kel} to $D(G)$. In fact Schneider works in a more general setup with a pair $(G,I)$ where $G$ is any $p$-adic Lie group (not necessarily compact) and $I \subset G$ is a compact open subgroup which is torsionfree and pro-$p$. Here we will only be interested in the case $G=I$. The results in this section are entirely due to Schneider. We hope our exposition will help the reader to quickly grasp the overall strategy behind \cite{Sch15}.

To apply theorem \ref{kel} we first have to observe that $D(G)$ is an {\it{algebraic}} triangulated category. 

\begin{lem}\label{alg}\cite[Lem.~8]{Sch15}
$D(G)$ is algebraic\footnote{The same proof shows that the derived category of {\it{any}} Grothendieck category is algebraic.}.
\end{lem}

\begin{proof}
Using Krause's lemma \ref{kra} this follows rather immediately from the existence of 
$K$-injective resolutions. Recall that a complex $V^{\bullet}$ is said to be $K$-injective if 
$$
\Hom_{K(G)}(U^{\bullet},V^{\bullet}) \overset{\sim}{\longrightarrow} \Hom_{D(G)}(U^{\bullet},V^{\bullet})
$$
for all complexes $U^{\bullet}$ (see \cite[Prop.~1.5]{Spa88}). Since $\Rep_k^{\infty}(G)$ is a Grothendieck category (meaning filtered colimits are exact and it has a generator $Y$) 
every complex $V^{\bullet}$ has a $K$-injective resolution, cf. \cite[Thm.~19.12.6]{Sta18}. That is, there is a $K$-injective complex $I^{\bullet}$ and a quasi-isomorphism $V^{\bullet}\rightarrow I^{\bullet}$. The $K$-injective complexes form a full triangulated subcategory $K_{\text{inj}}(G)$ of the homotopy category (see \cite[Prop.~1.3]{Spa88}) and there is a triangle equivalence
$$
K_{\text{inj}}(G) \overset{\sim}{\longrightarrow} D(G).
$$
Once and for all we fix a quasi-inverse ${\bf{i}}$. This shows that there is a fully faithful exact functor $D(G)\rightarrow K(G)$ and therefore $D(G)$ is algebraic by lemma \ref{kra}.
\end{proof}

It is easy to see that $D(G)$ has arbitrary coproducts. They are just direct sums of complexes, cf. \cite[Rem.~2]{Sch15}. The next step is to show the trivial representation is a compact generator for $D(G)$.
Occasionally when we write $k$ below we mean the trivial representation $G \overset{1}{\longrightarrow} k^{\times}$ concentrated in degree $0$. When there is a risk of confusion
we will write $k[0]$ instead.

\begin{lem}\label{cpt}\cite[Lem.~4]{Sch15}
The trivial representation $k$ is a compact object of $D(G)$. That is, the natural map
$$
\bigoplus_{j\in J} \Hom_{D(G)}(k,V_j^{\bullet}) \overset{\sim}{\longrightarrow} \Hom_{D(G)}(k,\bigoplus_{j\in J} V_j^{\bullet})
$$
is an isomorphism of abelian groups for any collection of complexes $(V_j^{\bullet})_{j \in J}$. 
\end{lem}

\begin{proof}
We have to show the functor $\Hom_{D(G)}(k,-)$ commutes with direct sums. Note that for all $V^{\bullet}$
\begin{equation}\label{comp}
\begin{split}
\Hom_{D(G)}(k,V^{\bullet})&=\Hom_{D(G)}\big(k,{\bf{i}}(V^{\bullet})\big)\\
  &=\Hom_{K(G)}\big(k,{\bf{i}}(V^{\bullet})\big)\\
  &=h^0\big(\Hom_{G}^{\bullet}(k[0],{\bf{i}}(V^{\bullet})\big)\\
  &=h^0\big(\Hom_{G}(k,{\bf{i}}(V^{\bullet})\big)\\
  &=h^0\big({\bf{i}}(V^{\bullet})^G \big)\\
  &=h^0(G,V^{\bullet}).
\end{split}
\end{equation}
Here $h^i(G,V^{\bullet})$ denotes group hypercohomology, by which we mean the following. Consider the additive functor 
$\Gamma: \Rep_k^{\infty}(G) \rightarrow \text{Vec}_k$ taking $G$-invariants $(\cdot)^G$. Since $R^i\Gamma=0$ for $i>d=\text{cd}_p(G)$ the total derived functor
$R\Gamma$ is defined on all of $D(G)$ by \cite[Cor.~5.3($\gamma$)]{Har66} and is given by the composition
$$
D(G)\overset{{\bf{i}}}{\longrightarrow} K_{\text{inj}}(G) \longrightarrow K(\text{Vec}_k) \longrightarrow D(\text{Vec}_k)
$$
which takes $V^{\bullet}\mapsto {\bf{i}}(V^{\bullet})^G$. Group hypercohomology is then $h^i(G,V^{\bullet})=h^i\big(R\Gamma(V^{\bullet})\big)$. Note that negative $i$ are allowed 
here. Since (actual) group cohomology commutes with direct sums in $\Rep_k^{\infty}(G)$ a direct sum of $\Gamma$-acyclic objects is $\Gamma$-acyclic. After choosing 
quasi-isomorphisms $V_j^{\bullet}\rightarrow C_j^{\bullet}$ with complexes having $\Gamma$-acyclic terms it is then easy to see that 
$$
\bigoplus_{j\in J} h^i(G,V_j^{\bullet}) \overset{\sim}{\longrightarrow}h^i(G,\bigoplus_{j\in J} V_j^{\bullet}).
$$
With the observation (\ref{comp}) the case $i=0$ gives the compactness of $k$. See the proof of \cite[Lem.~3]{Sch15} for more details.
\end{proof}

The calculation (\ref{comp}) shows more generally that -- for any $V^{\bullet}$ and any integer $i \in \Z$ -- we have
$$
\Hom_{D(G)}(k,T^iV^{\bullet})=h^i(G,V^{\bullet}).
$$
The most difficult part of \cite{Sch15} is the proof that $k$ is a generator of $D(G)$. We find it worthwhile to paraphrase Schneider's elegant proof below; our rendition will hopefully be helpful.

\begin{prop}\label{gen}\cite[Prop.~5]{Sch15}
The trivial representation $k$ is a generator of $D(G)$. That is
$$
\Hom_{D(G)}(k,T^iV^{\bullet})=0 \y \forall i \in \Z  \Longrightarrow V^{\bullet}=0.
$$
\end{prop}

\begin{proof}
Our assumption that $h^i(G,V^{\bullet})=0$ for all $i \in \Z$ means $R\Gamma(V^{\bullet})$ is exact. We have to conclude $V^{\bullet}$ is exact.
We will introduce a sequence of functors $\Gamma_j$ starting from $\Gamma_1=\Gamma$ which all admit total derived functors
$R\Gamma_j$ defined on $D(G)$. By induction we will show that $R\Gamma_j(V^{\bullet})$ is exact for all $j \in \N$. If the terms of $V^{\bullet}$ are $\Gamma_j$-acyclic
this means $\Gamma_j(V^{\bullet})$ is exact and we would be able to conclude $V^{\bullet}=\varinjlim_{j \in \N}\Gamma_j(V^{\bullet})$ is exact.

The forthcoming definition of the functors $\Gamma_j$ is based on the following observation. We first note that
$\Omega/I_N\simeq \ind_N^G(1)$ is an object of $\Rep_k^{\infty}(G)$ and $\Hom_G(\Omega/I_N,V)=V^N$ for any smooth representation $V$ by Frobenius reciprocity. Now since 
$\m^j$ is open it contains $I_N$ for small enough $N$ and therefore $\Omega/\m^j$ is also an object of $\Rep_k^{\infty}(G)$. We define the functors
$\Gamma_j: \Rep_k^{\infty}(G) \rightarrow \text{Vec}_k$ by 
$$
\Gamma_j(V)=\Hom_G(\Omega/\m^j,V).
$$
Note that $\Gamma_1=\Gamma$. We may identify $\Gamma_j(V)$ with a subspace of $V^N$ when $\m^j \supset I_N$. Since the powers $\m^j$ form a fundamental system of open neighborhoods at zero in $\Omega$, smoothness of $V$ translates into 
$$
V=\bigcup_{j \in \N} \Gamma_j(V).
$$ 
In order to show $R\Gamma_j$ is defined on all of $D(G)$ it suffices (again by \cite[Cor.~5.3($\gamma$)]{Har66}) to show that 
\begin{equation}\label{cd}
R^i\Gamma_j(V)=\Ext_G^i(\Omega/\m^j,V)=0 \y \y \forall i>d.
\end{equation}
Since $\Omega$ is Noetherian $\m^j/\m^{j+1}\simeq k^{\oplus n_j}$ for some $n_j\in \N$ (with $G$ acting trivially on both sides of the isomorphism). The short exact sequence
of smooth $G$-representations
$$
0 \longrightarrow \m^j/\m^{j+1} \longrightarrow \Omega/\m^{j+1} \longrightarrow \Omega/\m^j \longrightarrow 0
$$
gives rise to a long exact sequence of $k$-vector spaces
$$
\cdots \longrightarrow \Ext_G^i(\Omega/\m^j,V) \longrightarrow \Ext_G^i(\Omega/\m^{j+1},V)\longrightarrow \Ext_G^i(k,V)^{\oplus n_j} \longrightarrow \Ext_G^{i+1}(\Omega/\m^j,V)\longrightarrow \cdots
$$
from which (\ref{cd}) follows immediately by induction on $j$. The long exact sequence also shows that if $V$ is $\Gamma$-acyclic then $V$ is $\Gamma_j$-acyclic for all $j \in \N$. 

We proceed to show $R\Gamma_j(V^{\bullet})$ is exact (the case $j=1$ being our assumption on $V^{\bullet}$). The analogous long exact sequence for hyperext functors is
$$
\cdots \longrightarrow \Ext_G^i(\Omega/\m^j,V^{\bullet}) \longrightarrow \Ext_G^i(\Omega/\m^{j+1},V^{\bullet})\longrightarrow \Ext_G^i(k,V^{\bullet})^{\oplus n_j} \longrightarrow \Ext_G^{i+1}(\Omega/\m^j,V^{\bullet})\longrightarrow \cdots
$$
see \cite[Prop.~6.1]{Har66}. By induction all terms here vanish. A calculation similar to (\ref{comp}) shows that 
$$
\Ext_G^i(\Omega/\m^j,V^{\bullet}):=\Hom_{D(G)}(\Omega/\m^j,T^iV^{\bullet})=h^i(R\Gamma_j(V^{\bullet})).
$$
We may assume all terms of $V^{\bullet}$ are $\Gamma$-acyclic, and therefore $\Gamma_j$-acyclic as observed. As a result $\Gamma_j(V^{\bullet})$ is exact. Their colimit over $j$ 
equals $V^{\bullet}$ which is therefore exact.
\end{proof}

Using results of Neeman, Schneider deduces that $k$ generates $D(G)$ in the sense that any strict full triangulated subcategory $\mathcal{D}\subset D(G)$, closed under direct sums, which contains $k$ must be all of $D(G)$, cf. \cite[Prop.~6]{Sch15}. 

By contrast $k$ does not generate $\Rep_k^{\infty}(G)$ as an abelian category. It generates the subcategory of representations $V$ which are generated by $V^G$. Note that $V^G\neq 0$ for nonzero $V$ (pick a nonzero $v \in V$ and decompose the finite set $\F_p[G]v$ into $G$-orbits).

Combining the previous three results (Lemmas \ref{alg} and \ref{cpt}, and Proposition \ref{gen}) Keller's theorem \ref{kel} implies the main result of \cite{Sch15} in our setup ($G=I$):

\begin{thm}\label{main}\cite[Thm.~9]{Sch15}
There is an equivalence of triangulated categories
$$
D(G) \overset{\sim}{\longrightarrow} D(\mathcal{H}^{\bullet}).
$$
\end{thm}

Here $\mathcal{H}^{\bullet}=\End_G^{\bullet}(I^{\bullet})^{\op}$ for a choice of injective resolution $k[0]\rightarrow I^{\bullet}$ (which one can obviously take to be non-negatively graded)
and the equivalence is given by the composition
$$
H: D(G) \overset{{\bf{i}}}{\longrightarrow} K_{\text{inj}}(G) \longrightarrow K(\mathcal{H}^{\bullet}) \longrightarrow D(\mathcal{H}^{\bullet}) \y \y \y \y V^{\bullet} \rightsquigarrow \Hom_G^{\bullet}(I^{\bullet},{\bf{i}}(V^{\bullet})).
$$
A quasi-inverse can be defined in a similar fashion. Following \cite[Def.~10.12.2.1]{BL94} we call a DG module $M^{\bullet}$ $K$-projective
if the natural map $\Hom_{K(\mathcal{H}^{\bullet})}(M^{\bullet},N^{\bullet}) \longrightarrow \Hom_{D(\mathcal{H}^{\bullet})}(M^{\bullet},N^{\bullet})$ is an isomorphism for all other DG modules $N^{\bullet}$. By \cite[Cor.~10.12.2.9]{BL94} localization restricts to an equivalence of triangulated categories $K_{\text{pro}}(\mathcal{H}^{\bullet}) \longrightarrow D(\mathcal{H}^{\bullet})$. We fix a quasi-inverse ${\bf{p}}$ once and for all and let $T$ be the composition
$$
T: D(\mathcal{H}^{\bullet}) \overset{{\bf{p}}}{\longrightarrow} K_{\text{pro}}(\mathcal{H}^{\bullet}) \longrightarrow K(G) \longrightarrow D(G) \y \y \y \y
M^{\bullet} \rightsquigarrow I^{\bullet}\otimes_{\mathcal{H}^{\bullet}}{\bf{p}}(M^{\bullet}).
$$
Here we view $I^{\bullet}$ as the right DG $\mathcal{H}^{\bullet}$-module $\bigoplus_{i\geq 0}I^i$. The tensor product is defined in \cite[10.9]{BL94}. This functor $T$ is both a left adjoint to $H$ and a quasi-inverse of $H$. 

As noted after theorem \ref{kel} the cohomology algebra is $h^*(\mathcal{H}^{\bullet})=\Ext_G^*(k,k)^{\op}$. Moreover, 
$$
h^*\big(H(V^{\bullet})\big)=\bigoplus_{i \in \Z} \Ext_G^i(k,V^{\bullet})
$$
for any complex of smooth $G$-representations $V^{\bullet}$.

\section{The Yoneda algebra of $\Omega$}\label{yon}

We defined the DGA $\mathcal{H}^{\bullet}=\End_G^{\bullet}(I^{\bullet})^{\op}$ by choosing an injective resolution of the trivial representation;
say $k[0] \rightarrow I^{\bullet}$. By Pontryagin duality this corresponds to choosing a projective resolution $P^{\bullet}\rightarrow k[0]$ in the category $\Mod(\Omega)$ of pseudocompact modules. Here we view $k$ as a module over $\Omega$ via the augmentation map $\Omega \overset{a}{\longrightarrow} k$ and $P^{\bullet}$ is indexed non-positively:
$$
\cdots \longrightarrow P^{-2}\longrightarrow P^{-1}\longrightarrow P^0 \longrightarrow k \longrightarrow 0.
$$
To go between the two points of view use the relation $P^{-i}=(I^i)^{\vee}$ for $i\geq 0$. 

\begin{lem}\label{gomega}
We have the following isomorphisms.
\begin{itemize}
\item[(1)] $\mathcal{H}^{\bullet}=\End_G^{\bullet}(I^{\bullet})^{\op} \overset{\sim}{\longrightarrow} \End_{\Omega}^{\bullet}(P^{\bullet})$;
\item[(2)] $h^*(\mathcal{H}^{\bullet})=\Ext_G^*(k,k)^{\op}  \overset{\sim}{\longrightarrow} \Ext_{\Omega}^*(k,k)$.
\end{itemize}
\end{lem}

\begin{proof}
Consider a homogeneous element $a=(a_q)_{q\in \Z}$ in $\End_G^{i}(I^{\bullet})$. Here $a_q: I^q\rightarrow I^{q+i}$ has Pontryagin dual
$a_q^{\vee}: P^{-q-i}\rightarrow P^{-q}$. Introducing $t=-q-i$ this is a map $b_t: P^t \rightarrow P^{t+i}$. Packaged together they give an element $b=(b_t)_{t\in \Z}$ in
$\End_{\Omega}^{i}(P^{\bullet})$. To summarize $b_t=a_{-t-i}^{\vee}$. Since $(\cdot)^{\vee}$ is an anti-equivalence composition on $\End_{\Omega}^{\bullet}(P^{\bullet})$ corresponds to opposite composition on $\End_G^{\bullet}(I^{\bullet})$. This proves (1) which then implies (2) by taking cohomology.
\end{proof}

There is a canonical choice of $P^{\bullet}$ which is the {\it{bar}} resolution of $k$, cf. \cite[Ch.~V(1.2.8)]{Laz65}. In this resolution the $P^{-i}$ are $(i+1)$-fold completed tensor products (thought of as $\Omega$-modules via the leftmost factor)
$$
P^{-i}=\Omega^{\hat{\otimes}(i+1)}=\Omega \hat{\otimes}_k \cdots \hat{\otimes}_k \Omega=\Omega(G \times \cdots\times G)
$$
and the differentials are obtained from the case of $k[G/N]$ by passing to the limit, cf. \cite[Ch.~V(1.2.1.2)]{Laz65} for the explicit formula. The first differential
$P^0=\Omega\rightarrow k$ is the augmentation map $a$. Dually the $I^i$ are direct sums of (an enormous number of) copies of the regular representation $\mathcal{C}^{\infty}(G,k)$.

In conjunction Schneider's theorem \ref{main} and Pontryagin duality (corollary \ref{pd}) give an equivalence 
$$
D(\Omega)^{\op} \overset{\sim}{\longrightarrow} D(\End_{\Omega}^{\bullet}(P^{\bullet}))
$$
which takes a $K$-projective complex $X^{\bullet}$ of pseudocompact $\Omega$-modules to the left DG $\End_{\Omega}^{\bullet}(P^{\bullet})$-module
$\Hom_{\Omega}^{\bullet}(X^{\bullet},P^{\bullet})$. After passing to cohomology this yields a left module $\bigoplus_{i\in \Z}\Ext_{\Omega}^i(X^{\bullet},k)$ over the Yoneda algebra
$$
E(\Omega):=\Ext_{\Omega}^*(k,k).
$$
One of our goals in this paper is to endow $E(\Omega)$ with an $A_{\infty}$-structure and give an equivalence between $D(\Omega)^{\op}$ and the derived category of
$A_{\infty}$-modules $D_{\infty}(E(\Omega))$ akin to Koszul duality. We pause before doing so and discuss the case of equi-$p$-valued groups $G$ where $E(\Omega)$ has a nice explicit description due to Lazard. 

\section{Review of Lazard's theory of $p$-valued groups}\label{valu}

\subsection{Filtrations on $p$-valued groups}

Suppose our pro-$p$ group $G$ is equipped with a $p$-valuation $\omega$ defining the topology. This means the following. First recall that 
$\omega: G\backslash \{1\}\longrightarrow (\frac{1}{p-1},\infty)$ is a function satisfying three axioms, one of which is that
$\omega(g^p)=\omega(g)+1$, cf. \cite[p.~169]{Sch11}. In particular $G$ cannot have elements of order $p$. The other conditions on $\omega$ guarantee that 
$$
G_v=\{g\in G: \omega(g)\geq v\} \y \y \y \y \y G_{v+}=\{g\in G: \omega(g)>v\}
$$
are normal subgroups of $G$. When we say $\omega$ defines the topology on $G$ we mean that each $G_v$ is open and $G=\varprojlim G/G_v$. We associate an abelian group 
$\gr(G)$ with the $p$-valued group $(G,\omega)$ as follows.
$$
\gr(G)=\bigoplus_{v>0}\gr_v(G)=\bigoplus_{v>0} G_v/G_{v+}.
$$
It turns out $\gr(G)$ has a lot more structure. First of all it is an $\F_p$-vector space since $g^pG_{v+}=1G_{v+}$. Moreover, taking commutators define a Lie bracket 
$[-,-]$ on $\gr(G)$ compatible with the grading. Lastly, $\gr(G)$ becomes a module over the one-variable polynomial ring $\F_p[\pi]$ by letting $\pi$ act by the degree one map
$$
\pi: \gr(G)\longrightarrow \gr(G) \y \y \y \y \y gG_{v+}\mapsto g^pG_{(v+1)+}.
$$
One can show that the Lie bracket is $\F_p[\pi]$-bilinear and $\gr(G)$ thus becomes a graded Lie algebra over $\F_p[\pi]$. We refer to \cite[Sect.~23-25]{Sch11} for more details.
If $k$ is a field of characteristic $p$ we convert $\gr(G)$ into a graded Lie algebra $\frak{g}$ over $k$ as follows. This $\frak{g}$ will play an important role below. 

\begin{defn}
$\frak{g}:=k \otimes_{\F_p[\pi]} \gr(G)=k\otimes_{\F_p} \gr(G)/\pi\gr(G)$.
\end{defn}

In general $\gr(G)$ is torsionfree over $\F_p[\pi]$. Since we are assuming $G$ is a $d$-dimensional Lie group over $\Q_p$ in fact $\gr(G)$ is free over $\F_p[\pi]$ of rank $d$.
Consequently $(G,\omega)$ admits an ordered basis; by which we mean a tuple $(g_1,\ldots,g_d)$ of elements from $G$ with the two properties below. 

\begin{itemize}
\item The map $(x_1,\ldots,x_d) \mapsto g_1^{x_1}\cdots g_d^{x_d}$ is a homeomorphism $\Z_p^d \longrightarrow G$;
\item $\omega(g)=\min_{i=1,\ldots,d} \{\omega(g_i)+v(x_i)\}$ where $g=g_1^{x_1}\cdots g_d^{x_d}$.
\end{itemize}

Correspondingly the elements $\sigma(g_i):=g_iG_{\omega(g_i)+}$ form an $\F_p[\pi]$-basis for $\gr(G)$, cf. \cite[Prop.~26.5]{Sch11}. In turn
the vectors $\xi_i:=1\otimes \sigma(g_i)$ form a $k$-basis for $\frak{g}$, and by the Poincar\'{e}-Birkhoff-Witt theorem the monomials
$\xi_1^{\alpha_1}\cdots\xi_d^{\alpha_d}$ (for varying $\alpha_i \in \Z_{\geq 0}$) form a $k$-basis for the universal enveloping algebra $U(\frak{g})$. 

\subsection{Filtrations on Iwasawa algebras}

Now assume $k$ is a {\it{finite}} field of characteristic $p$ and let $\OO=W(k)$ be its ring of Witt vectors; a complete DVR with residue field 
$\OO/p\OO=k$. We let $\Lambda=\Lambda(G)=\OO[\![G]\!]$ be the Iwasawa $\OO$-algebra of $G$. Note that $\Lambda/p\Lambda=\Omega$ (the Mittag-Leffler condition is trivially satisfied since the transition maps between the various $\OO[G/N]$ are surjective). 

As an $\OO$-module $\Lambda$ is isomorphic to $\OO[\![X_1,\ldots,X_d]\!]$. If we let ${\bf{b}}_i:=g_i-1$ and ${\bf{b}}^{\alpha}:={\bf{b}}_1^{\alpha_1}\cdots {\bf{b}}_d^{\alpha_d}$ 
for $\alpha=(\alpha_i)\in \Z_{\geq 0}^d$ an isomorphism is given by 
$$
\OO[\![X_1,\ldots,X_d]\!] \overset{\sim}{\longrightarrow} \Lambda \y \y \y \y \y \sum_{\alpha} c_{\alpha}X^{\alpha} \mapsto \sum_{\alpha} c_{\alpha}{\bf{b}}^{\alpha}
$$
where $X^{\alpha}$ is defined in the obvious way. This leads to a valuation $\tilde{\omega}: \Lambda\backslash \{0\} \longrightarrow (\frac{1}{p-1},\infty)$ on the Iwasawa $\OO$-algebra 
given by the formula
$$
\tilde{\omega}(\lambda)=\inf_{\alpha}\{v(c_{\alpha})+\sum_{i=1}^d \alpha_i \omega(g_i)\} \y \y \y \y \y \lambda=\sum_{\alpha} c_{\alpha}{\bf{b}}^{\alpha}
$$
where $v(p)=1$. Since $\omega(G \backslash \{1\})$ is a discrete subset of $(0,\infty)$ one may replace the infimum by the minimum in the definition of 
$\tilde{\omega}$. Analogous to $G_v$ and $G_{v+}$ we define $\OO$-submodules of $\Lambda$ as 
$$
\Lambda_v=\{\lambda\in \Lambda: \tilde{\omega}(\lambda)\geq v\} \y \y \y \y \y \Lambda_{v+}=\{\lambda\in \Lambda: \tilde{\omega}(\lambda)>v\}.
$$
They turn out to be two-sided ideals of $\Lambda$ which form a fundamental system of open neighborhoods of zero in $\Lambda$ (as $v>0$ varies). Recall that $\Lambda$ is a local ring since $G$ is pro-$p$, and its pseudocompact topology coincides with the $\m_{\Lambda}$-adic topology where $\m_{\Lambda}=\{\lambda: v(c_0)>0\}$. Each $\Lambda_v$ has a more explicit description (cf. \cite[p.~197]{Sch11}) which in particular shows that $\Lambda_v$ is independent of the choice of basis $(g_i)$; but dependent on $\omega$ of course. 
The same is true for $\tilde{\omega}$, cf. \cite[Cor.~28.4]{Sch11}. Moreover multiplication takes $\Lambda_v \times \Lambda_{v'}\rightarrow \Lambda_{v+v'}$, which allows us to associate a graded ring $\gr(\Lambda)$ defined as
$$
\gr(\Lambda)=\bigoplus_{v>0}\gr_v(\Lambda)=\bigoplus_{v>0} \Lambda_v/\Lambda_{v+}.
$$ 
Note that $\gr(\Lambda)$ is a graded algebra over 
$$
\gr(\OO)=\bigoplus_{i\geq 0} p^i\OO/p^{i+1}\OO.
$$
(When $k=\F_p$ we have an isomorphism $\F_p[\pi] \overset{\sim}{\longrightarrow} \gr(\Z_p)$ sending $\pi\mapsto p+p^2\Z_p$.) We refer to \cite[Sect.~ 28]{Sch11} for more details on these constructions and proofs of the facts mentioned. 

One of the main results in Lazard theory identifies $\gr(\Lambda)$ with the universal enveloping algebra $U(\gr(G))$ tensored by $\gr(\OO)$. 

\begin{thm}\label{laz}\cite[Thm.~28.3.i]{Sch11}
There is an isomorphism of graded $\gr(\OO)$-algebras
$$
\gr(\OO)\otimes_{\F_p[\pi]} U(\gr(G)) \overset{\sim}{\longrightarrow} \gr(\Lambda).
$$
\end{thm}

This isomorphism is constructed as follows. Taking $gG_{v+}\mapsto (g-1)+\Lambda_{v+}$ is a group homomorphism 
$\gr_v(G) \rightarrow \gr_v(\Lambda)$. Varying $v$ these add up to an $\F_p[\pi]$-linear map $\gr(G)\rightarrow \gr(\Lambda)$. One checks that the Lie bracket $[-,-]$ on  
$\gr(G)$ intertwines with taking commutators on $\gr(\Lambda)$ via this map, which therefore by universality extends to $U(\gr(G))$. 

Similar results hold for $\Omega$. Its elements can be uniquely expanded as $\sum_{\alpha}c_{\alpha}{\bf{b}}^{\alpha}$ now with coefficients $c_{\alpha}\in k$ of course. 
We define $\Omega_v$ to be the image of $\Lambda_v$ under the projection $\Lambda \twoheadrightarrow \Omega$. Similarly for $\Omega_{v+}$. 
$$
\gr(\Omega):=\bigoplus_{v>0}\gr_v(\Omega)=\bigoplus_{v>0} \Omega_v/\Omega_{v+}
$$
is then a graded $k$-algebra. Observe that $\gr(\OO)=k\oplus (p\OO/p^2\OO)\oplus\cdots$ has a natural augmentation map $\gr(\OO)\rightarrow k$, whose kernel we will denote 
by $\gr_+(\OO)$.

\begin{cor}\label{lazmod}
There is an isomorphism of graded $k$-algebras
$$
U(\frak{g})=k\otimes_{\F_p[\pi]} U(\gr(G)) \overset{\sim}{\longrightarrow} \gr(\Omega).
$$
\end{cor}

\begin{proof}
This follows immediately from theorem \ref{laz} once we check that $k \otimes_{\gr(\OO)}\gr(\Lambda)\overset{\sim}{\longrightarrow} \gr(\Omega)$. The projection
$\gamma: \Lambda\twoheadrightarrow \Omega$ has kernel $p\Lambda$ and induces surjections $\Lambda_v \twoheadrightarrow \Omega_v$ (and similarly for the $+$ versions). 
We get an induced map $\bar{\gamma}: \gr(\Lambda)\twoheadrightarrow \gr(\Omega)$ which is clearly trivial on $\gr_+(\OO)\cdot \gr(\Lambda)$. Indeed $\forall i>0$
$$
\bar{\gamma}\big( (p^ic+p^{i+1}\OO)(\lambda+\Lambda_{v+})\big)=\bar{\gamma}\big(p^ic\lambda +\Lambda_{(i+v)+}\big)=\gamma(p^ic\lambda)+\Omega_{(i+v)+}=
0.
$$
One can easily reverse this argument and show the factored map $k \otimes_{\gr(\OO)}\gr(\Lambda)\twoheadrightarrow \gr(\Omega)$ is injective. For instance, in the homogeneous case, if $\bar{\gamma}(\lambda+\Lambda_{v+})=0$ then $\gamma(\lambda)\in \Omega_{v+}$ which means $\gamma(\lambda)=\gamma(\lambda')$ for some
$\lambda'\in \Lambda_{v+}$. Since $\ker(\gamma)=p\Lambda$ we infer that $\lambda=p\eta +\lambda'$ for some $\eta \in \Lambda_{v-1}$. 
\end{proof}

Unwinding the definitions shows the isomorphism in corollary \ref{lazmod} sends $\xi_i \mapsto {\bf{b}}_i+\Omega_{\omega(g_i)+}$. More generally it takes
$$
\xi^{\alpha} \mapsto {\bf{b}}^{\alpha} + \Omega_{(\sum_{i=1}^d \alpha_i \omega(g_i))+}
$$
where $\xi^{\alpha}:=\xi_1^{\alpha_1}\cdots \xi_d^{\alpha_d}$ is the PBW basis for $U(\frak{g})$ introduced earlier. The grading $U(\frak{g})=\bigoplus_{v>0}\gr_v U(\frak{g})$ implicit in \ref{lazmod} is defined by taking $\gr_v U(\frak{g})$ to be the $k$-span of all $\xi^{\alpha}$ for which $\sum_{i=1}^d \alpha_i \omega(g_i)=v$.

We will mostly be interested in the case where $\frak{g}$ is an abelian Lie algebra. If so $U(\frak{g})$ is the symmetric algebra $S(\frak{g})$, and $\gr(\Omega)$ is a polynomial 
algebra $k[X_1,\ldots,X_d]$ with grading determined by decreeing that $\deg(X_i)=\omega(g_i)$ for all $i$. If these are all one this is of course the usual polynomial degree.

\section{Mod $p$ cohomology of equi-$p$-valued groups}

We say $(G,\omega)$ is equi-$p$-valued if it admits an ordered basis $(g_i)$ all of whose elements have the same valuation
$\omega(g_i)=t$. We will often assume $(G,\omega)$ is saturated (meaning all $g \in G$ with $\omega(g)>\frac{p}{p-1}$ are $p$-powers). This happens precisely when $t \leq 
\frac{p}{p-1}$, cf. \cite[Prop.~26.11]{Sch11}. In this case, when $p>2$ one can replace $\omega$ with $\omega'=\omega+1-t$ and arrange for $\omega'(g_i)=1$ to hold for all $i$. In other words if $p>2$ we may often assume that $t=1$. (We need $p>2$ for $\frac{1}{p-1}<1$ to hold.). See Weigel's remarks in \cite[App.~A]{HKN11} for more details on the switch from $\omega$ to $\omega'$.

When $(G,\omega)$ is equi-$p$-valued the $\F_p[\pi]$-basis elements $\sigma(g_i)=g_i G_{t+}$ all lie in $\gr_t(G)$. In particular $\gr_t(G)$ generates $\gr(G)$
as an $\F_p[\pi]$-module. Consequently $\frak{g}=\oplus_{v>0}\frak{g}_v=\frak{g}_t$ is concentrated in degree $t$ and is therefore an {\it{abelian}} Lie algebra; $[\frak{g}_t,\frak{g}_t]\subset \frak{g}_{2t}=0$. In particular $\gr(\Omega)\simeq S(\frak{g})$ is a symmetric algebra. Having fixed the basis $(\xi_i)$ we may identify $S(\frak{g})\simeq k[X_1,\ldots,X_d]$
but with each $X_i$ of degree $t$. 

\begin{rem}
A good supply of examples arise as follows. Let $F/\Q_p$ be a finite extension with integers $\OO_F$ and residue field
$k_F=\OO_F/\m_F$. We let $v$ be the valuation on $F$ with $v(p)=1$. An $n$-variable formal group law $\mathcal{F}(\underline{X},\underline{Y})$ over $\OO_F$ gives a group structure on $\m_F^n=\m_F \times \cdots \times \m_F$ ($n$ factors). We denote the resulting group by $\tilde{G}$ (or $\tilde{G}_{\mathcal{F}}$ if need should arise to emphasize the formal group law). These are known as Serre's standard groups. We define a function $\omega: \tilde{G} \backslash \{1\} \rightarrow (0,\infty)$ by 
$$
\omega(x_1,\ldots,x_n)={\min}_{i=1,\ldots,n} v(x_i).
$$
This is of course not a $p$-valuation on $\tilde{G}$ itself unless $e(F/\Q_p)<p-1$, but one can show (cf. \cite[Lem.~2.2.2]{HKN11}) that $\omega$ does define a $p$-valuation on the normal subgroup
$$
G=\{g \in \tilde{G}: \omega(g)>\frac{1}{p-1}\}.
$$
Moreover $(G,\omega)$ is saturated, of rank $d=n[F:\Q_p]$, and equi-$p$-valued if and only if $F/\Q_p$ is unramified; in which case 
a $\Z_p$-basis for $\OO_F$ yields an ordered basis for $G$ whose elements all have valuation $t=1$. 

For instance one can start from a smooth affine group scheme $\mathcal{G}_{/\OO_F}$ of finite type, and consider its formal completion $\hat{\mathcal{G}}$ along the unit
$\Spec(\OO_F)\rightarrow \mathcal{G}$. This is a smooth formal group $\hat{\mathcal{G}}\simeq \text{Spf} (\OO_F[\![X_1,\ldots,X_n]\!])$, see \cite[p.~41]{Dem86}. Therefore
$\hat{\mathcal{G}}$ defines an $n$-variable formal group law $\mathcal{F}(\underline{X},\underline{Y})$ over $\OO_F$ to which we can associate a $p$-valued group
$(G,\omega)$ as in the previous paragraph. Unwinding the definitions it is easy to check that
$$
\tilde{G} \simeq \ker\big(\mathcal{G}(\OO_F)\longrightarrow \mathcal{G}(k_F)\big).
$$
Note that $\tilde{G}=G$ when $e(F/\Q_p)=1$ and $p>2$.

On the negative side one can show that when $p>3$ there is no way to equip $1+\m_D$ with a $p$-valuation for which it becomes equi-$p$-valued. Here $D/\Q_p$ is the division quaternion algebra and $\m_D \subset \OO_D$ the maximal two-sided ideal of the maximal order. Indeed one can compute the Betti numbers
$\dim_{\F_p}H^i(1+\m_D,\F_p)$ to be $(1,3,4,3,1)$ which is not a row of Pascal's triangle, cf theorem \ref{coho} below. We found this example in \cite[Ex.~3.2.1]{HKN11}.
\end{rem}

The following important theorem is due to Lazard.

\begin{thm}\cite[Ch.~V Prop.~2.5.7.1]{Laz65}\label{coho}
Let $(G,\omega)$ be an equi-$p$-valued group of finite rank. Then there is a natural isomorphism of 
graded $\F_p$-algebras
$$
H^*(G,\F_p) \overset{\sim}{\longrightarrow} \bigwedge \Hom(G,\F_p).
$$
\end{thm}

We outline some of the key steps in the proof. Lazard starts from the Cartan-Eilenberg resolution $Y_{\bullet}$ of the trivial $U(\gr(G))$-module $\gr(\Z_p)=\F_p[\pi]$, cf. \cite[Thm.~7.1]{CE99}:
$$
\cdots \longrightarrow U(\gr(G))\otimes \bigwedge^2 \gr(G)\longrightarrow U(\gr(G))\otimes \bigwedge^1 \gr(G) \longrightarrow U(\gr(G))\longrightarrow \F_p[\pi] \longrightarrow 0.
$$
By a lemma of Serre one can lift $Y_{\bullet}$ to a complex $X_{\bullet}$ of $\Z_p[\![G]\!]$-modules, in the sense that $\gr X_{\bullet}=Y_{\bullet}$. It turns out that $X_{\bullet}$ is in fact a free resolution of $\Z_p$, which Lazard dubs the quasi-minimal resolution, cf. \cite[p.~165]{Laz65}. For equi-$p$-valued groups $(G,\omega)$ the quasi-minimal resolution is indeed minimal; meaning the differentials of $\F_p \otimes X_{\bullet}$ all vanish. In particular this allows one to immediately compute the Betti numbers
$\dim_{\F_p}H^i(G,\F_p)=\binom{d}{i}$ where $d=\rk(G)$. A more detailed analysis of $Y_{\bullet}$ and its counterpart for the mod $p$ Lie algebra $\frak{g}=\gr(G)/\pi \gr(G)$ (we take $k=\F_p$ here) yields an isomorphism of graded $\F_p$-algebras
$$
H^*(G,\F_p) \overset{\sim}{\longrightarrow} H^*(\frak{g},\F_p).
$$
Relating the cup products is a highly non-trivial step found in \cite[Ch.~V Sect.~2.5.6]{Laz65}. Since $\frak{g}$ is abelian its Lie algebra cohomology is 
$\bigwedge \frak{g}^{\vee}$ where $\frak{g}^{\vee}=\Hom_{\F_p}(\frak{g},\F_p)\simeq \Hom(G,\F_p)$. For the last identification note that homomorphisms $G \rightarrow \F_p$ factor through $G/G^p$ and as vector spaces $\gr(G)=(G/G^p)\oplus \pi \gr(G)$ when $(G,\omega)$ is saturated, cf. \cite[Lem.~26.10]{Sch11}.

For an arbitrary finite coefficient field $k \supset \F_p$ we immediately deduce the more general result (via the K\"{u}nneth formula for group cohomology say); still assuming 
$(G,\omega)$ is equi-$p$-valued of course.

\begin{cor}\label{cohom}
$H^*(G,k) \overset{\sim}{\longrightarrow} \bigwedge \Hom(G,k)\simeq \bigwedge \frak{g}^{\vee}$.
\end{cor}

Here the exterior and the dual are taken over $k$.

There is a close connection between equi-$p$-valued groups and {\it{uniform}} pro-$p$-groups, as studied in \cite[Part I.4]{DdSMS} for instance. The latter are topologically finitely generated pro-$p$-groups for which $G/\overline{G^p}$ is abelian (for $p>2$) and the lower $p$-series $G=G_1\geq G_2\geq \cdots $ has the property that the indices
$[G_i:G_{i+1}]=p^d$ are independent of $i$. Recall that the lower $p$-series is defined as $G_{i+1}=G_i^p[G_i,G]$. One can verify that $\omega(g)=\sup\{i \in \N: g \in G_i\}$ defines a $p$-valuation for which $(G,\omega)$ becomes equi-$p$-valued with $t=1$. Conversely, Weigel's remarks summarized in \cite[Prop.~A.1]{HKN11} show that a saturated equi-$p$-valued group is uniform. 

\section{$A_{\infty}$-algebras and their minimal models}\label{model}

The notion of an $A_{\infty}$-algebra ("strongly homotopy associative algebra") first arose in Stasheff's 1961 Princeton Ph.D. thesis \cite{Sta61}. In topology loop spaces carry a natural operation -- composition of loops -- which is only associative up to homotopy. In homological algebra $A_{\infty}$-algebras naturally arise as follows. Suppose $A^{\bullet}$ is a complex of vector spaces, $B^{\bullet}$ is a DGA, and there is a quasi-isomorphism $i:A^{\bullet}\rightarrow B^{\bullet}$ along with a morphism of chain complexes
$p: B^{\bullet}\rightarrow A^{\bullet}$ such that $ip\sim \text{Id}_{B^{\bullet}}$ ("homotopy retract"). Say $\text{Id}_{B^{\bullet}}-ip=dh+hd$ for some homotopy $h: B^{\bullet}\rightarrow B^{\bullet}$. When one transfers the multiplication on $B^{\bullet}$ via $i$ and $p$ to $A^{\bullet}$ the resulting operation $m_2: A^{\bullet}\otimes A^{\bullet}\rightarrow A^{\bullet}$ turns out to only be associative up to homotopy, in the following sense: The "associator" $m_2(m_2\otimes 1)-m_2(1\otimes m_2)$ lives in $\Hom(A^{\bullet \otimes 3},A^{\bullet})$ and one can explicitly write down an $m_3:A^{\bullet \otimes 3} \rightarrow A^{\bullet}$ in terms of $h$ for which 
$$
m_2(m_2\otimes 1)-m_2(1\otimes m_2)=\partial(m_3)
$$ 
where $\partial$ is the differential on the morphism complex. Continuing this way one finds a whole sequence of higher multiplications $m_n: A^{\bullet \otimes n} \rightarrow A^{\bullet}$ (of degree $2-n$) such that
$$
{\sum}_{r+s+t=n:  1<s<n}(-1)^{rs+t+1} m_{r+1+t}(1^{\otimes r}\otimes m_s \otimes 1^{\otimes t})=\partial(m_n).
$$
Note that $m_1=d_{A^{\bullet}}$ and for $n=2$ the above identity is just the Leibniz rule. This idea of homotopy transfer is explained and illustrated very clearly in 
\cite[Ch.~1]{Val14}. In the 1980's Kadeishvili developed the algebraic point of view further and laid the foundations in a series of papers \cite{Kad83, Kad86, Kad88} which unfortunately are hard to find (and in Russian). Some of his results were later reproved by Merkulov in \cite{Mer99}. Kontsevich made $A_{\infty}$-algebras a hot topic in his 1994 ICM paper 
\cite{Kon95} in which he proposed a connection to mirror symmetry via the Fukaya $A_{\infty}$-category of a symplectic manifold. The subject gained momentum and remains a very active research area today. 

Keller has several excellent introductions to the subject \cite{Kel01,Kel02,Kel06}. Below we will follow the conventions of \cite{Kel01}.

\begin{defn}\label{Aalg}\cite[Def.~3.1]{Kel01}
An $A_{\infty}$-algebra over a field $k$ is a a graded vector space $A^{\bullet}=\bigoplus_{i\in \Z}A^i$ endowed with homogeneous $k$-linear maps
$$
m_n: A^{\bullet \otimes n}\longrightarrow A^{\bullet} \y \y \y \y \y n\geq 1
$$
such that $m_n$ has degree $2-n$ and this sequence of higher multiplications $(m_n)_{n\geq 1}$ satisfies the following properties.
\begin{itemize}
\item[(1)] $m_1 \circ m_1=0$ -- so $(A^{\bullet},m_1)$ is a complex;
\item[(2)] For $n\geq 1$ the Stasheff identities are satisfied. That is,
$$
\sum (-1)^{r+st}m_{r+1+t}(1^{\otimes r}\otimes m_s\otimes 1^{\otimes t})=0
$$
where the sum runs over all decompositions $r+s+t=n$ with $s \geq 1$ (and $r,t$ non-negative). 
\end{itemize}
(For $n=1$ the identity in (2) reduces to (1). For $n=2$ the identity shows $m_1$ is a derivation.)
\end{defn}

The Koszul sign convention is in force here when we apply $m_n$ to homogeneous elements (meaning when we swap $x$ and $y$ a sign $(-1)^{\deg(x)\deg(y)}$ appears). For example, for homogeneous $a, b \in A^{\bullet}$ we have 
$$
(m_1\otimes 1 + 1 \otimes m_1)(a\otimes b)=m_1(a)\otimes b+(-1)^{\deg(a)}a \otimes m_1(b)
$$
since we swap $m_1$ (of degree one) and $a$ in the second term.

Note that a DGA is the same thing as an $A_{\infty}$-algebra with $m_n=0$ for all $n \geq 3$. Furthermore the cohomology $h^*(A^{\bullet})$ of any $A_{\infty}$-algebra $A^{\bullet}$ 
is a graded algebra under the multiplication induced by $m_2$. We often view $h^*(A^{\bullet})$ as a {\it{minimal}} DGA (meaning it has differential $d=0$). One of the most fundamental results of the whole theory is the fact that $h^*(A^{\bullet})$ carries a natural $A_{\infty}$-structure. This is what we will discuss and utilize below. 

First we recall how morphisms are defined.

\begin{defn}\label{Amor}\cite[Def.~3.4]{Kel01}
A morphism $f: A^{\bullet}\rightarrow B^{\bullet}$ between $A_{\infty}$-algebras $A^{\bullet}$ and $B^{\bullet}$ is a collection of homogeneous $k$-linear maps
$$
f_n: A^{\bullet \otimes n}\longrightarrow B^{\bullet} \y \y \y \y \y n\geq 1
$$
such that $f_n$ has degree $1-n$ and the sequence $(f_n)_{n\geq 1}$ satisfies the morphism identity ($\forall n$)
$$
\sum (-1)^{r+st} f_{r+1+t}(1^{\otimes r}\otimes m_s^A\otimes 1^{\otimes t})=\sum (-1)^{\sigma} m_r^B(f_{i_1}\otimes \cdots \otimes f_{i_r}).
$$
Here the first sum runs over the same decompositions $r+s+t=n$ as in Definition \ref{Aalg}(2), and the second sum runs over 
$1\leq r\leq n$ and all decompositions $i_1+\cdots+i_r=n$. The sign $(-1)^{\sigma}$ is given by
$$
\sigma:=\sum_{j=1}^{r-1}(r-j)(i_j-1)=(r-1)(i_1-1)+(r-2)(i_2-1)+\cdots+(i_{r-1}-1).
$$
(For $n=1$ this says $f_1: A^{\bullet} \rightarrow B^{\bullet}$ is a morphism of complexes.) We say $f$ is a quasi-isomorphism if $f_1$ induces an isomorphism
on cohomology $h^*(A^{\bullet})\overset{\sim}{\longrightarrow} h^*(B^{\bullet})$.
\end{defn}

The identity morphism $\text{Id}: A^{\bullet}\rightarrow A^{\bullet}$ is the one with $\text{Id}_1=\text{Id}_{A^{\bullet}}$ and $\text{Id}_n=0$ for $n>1$, and there is a natural way to compose morphisms which we will not recall here (cf. \cite[p.~10]{Kel01}) which yields the category of $A_{\infty}$-algebras over $k$. 

\begin{rem}\label{sunit}
In the applications we have in mind our $A_{\infty}$-algebras will be strictly {\it{unital}}, which means $A^{\bullet}$ contains a two-sided multiplicative identity $1=1_{A^{\bullet}}\in A^0$ for $m_2$ with the property $m_n(a_1\otimes \cdots \otimes a_n)=0$ if $n\neq 2$ and some $a_i=1$. A morphism $f: A^{\bullet}\rightarrow B^{\bullet}$ between strictly unital $A_{\infty}$-algebras is said to be strictly unital if $f_1(1_{A^{\bullet}})=1_{B^{\bullet}}$ and $f_n(a_1\otimes \cdots \otimes a_n)=0$ if $n>1$ and some $a_i=1$. 
\end{rem}

The following key result of Kadeishvili is usually referred to as the minimal model theorem.

\begin{thm}\label{kade}\cite[Sect.~3.3]{Kel01}
Let $A^{\bullet}$ be an $A_{\infty}$-algebra. Then the cohomology $h^*(A^{\bullet})$ has an $A_{\infty}$-structure $(\mu_n)_{n\geq 1}$ with the following properties.
\begin{itemize}
\item[(1)] $\mu_1=0$ ("minimality") and $\mu_2$ is induced by the $m_2$ on $A^{\bullet}$;
\item[(2)] There is a quasi-isomorphism of $A_{\infty}$-algebras 
$$
f:h^*(A^{\bullet})\longrightarrow A^{\bullet}
$$ 
lifting the identity on $h^*(A^{\bullet})$; that is $h^i(f_1)=\text{Id}_{h^i(A^{\bullet})}$ for all $i\in \Z$.
\end{itemize}
This $A_{\infty}$-structure $(\mu_n)_{n\geq 1}$  is uniquely determined by (1) and (2) up to (non-canonical) isomorphism of $A_{\infty}$-algebras.
\end{thm}

Eventually we will apply this result to the DGA $\mathcal{H}^{\bullet}$ and get an $A_{\infty}$-structure on $h^*(\mathcal{H}^{\bullet})\simeq E(\Omega)$.

The simple idea behind theorem \ref{kade} is very similar to the homotopy transfer discussed in the first paragraph of this section: Say we start with an $A_{\infty}$-algebra $A^{\bullet}$, or even just a DGA. For each $i$ let $B^i\subset Z^i$ be the coboundaries and cocycles in $A^i$  and once and for all {\it{choose}} vector space complements 
$$
Z^i=B^i\oplus H^i \y \y \y \y \y A^i=Z^i\oplus C^i \y \y (=B^i\oplus H^i \oplus C^i).
$$
Let $i:H^{\bullet} \rightarrow A^{\bullet}$ be the inclusion map and $p:A^{\bullet} \rightarrow H^{\bullet}$ be the projection map. Clearly $pi=\text{Id}_{H^{\bullet}}$. On the other hand $ip\sim \text{Id}_{A^{\bullet}}$. To build a homotopy $h: A^{\bullet} \rightarrow A^{\bullet}$ such that $\text{Id}_{A^{\bullet}}-ip=dh+hd$ first observe that the composition
$$
C^{i-1}\longrightarrow A^{i-1} \overset{d}\longrightarrow B^i
$$
is an isomorphism. This is immediate. Let $\gamma^i: B^i \rightarrow C^{i-1}$ denote its inverse, and extend it to a map $h^i: A^i \rightarrow C^{i-1}$ by setting $h^i=0$ on 
$H^i \oplus C^i$. Unwinding these definitions it is easily checked that $d \circ h^i$ is the projection $A^i\rightarrow B^i$ and that $h^{i+1}\circ d$ is the projection 
$A^i \rightarrow C^i$. Thus indeed $\text{Id}_{A^{\bullet}}-ip=dh+hd$. Now one simply transfers the $A_{\infty}$-structure on $A^{\bullet}$ to $H^{\bullet}\simeq h^*(A^{\bullet})$ via $i$ and $p$. One can explicitly write down the higher multiplications $(\mu_n)_{n\geq 1}$ in terms of the homotopy $h$ and $(m_n)_{n \geq 1}$, cf. \cite[Thm.~2.2]{LPWZ} 
and \cite[Ex.~5.2.8]{WE18} for DGA's.
Same for the quasi-isomorphism $H^{\bullet}\rightarrow A^{\bullet}$, cf. \cite[Prop.~2.3]{LPWZ}.

\section{$A_{\infty}$-modules and the derived category}\label{dercat}

Let $A^{\bullet}$ be an $A_{\infty}$-algebra. A left $A_{\infty}$-module over $A^{\bullet}$ is a graded $k$-vector space $M^{\bullet}=\bigoplus_{i \in \Z}M^i$ with homogeneous 
$k$-linear maps
$$
\nu_n: A^{\bullet \otimes n-1}\otimes M^{\bullet}\longrightarrow M^{\bullet} \y \y \y \y \y n\geq 1,
$$
where $\nu_n$ has degree $2-n$ and which satisfy the analogues of the Stasheff identities (Def. \ref{Aalg}(2)):
$$
\sum (-1)^{r+st}\nu_{r+1+t}(1^{\otimes r}\otimes m_s\otimes 1^{\otimes t})=0.
$$
(When $t=0$ the term $\nu_{r+1+t}(1^{\otimes r}\otimes m_s\otimes 1^{\otimes t})$ is interpreted as $\nu_{r+1}(1^{\otimes r}\otimes \nu_s)$.) For $n=1$ this says
$\nu_1\circ\nu_1=0$. An $A_{\infty}$-module structure on a complex $(M^{\bullet},\nu_1)$ amounts to a morphism of $A_{\infty}$-algebras
$A^{\bullet}\rightarrow \End_k^{\bullet}(M^{\bullet})$ where the target is thought of as a DGA. In all our applications $A^{\bullet}$ will be strictly unital, cf. remark \ref{sunit}.
We will henceforth exclusively consider {\it{strictly unital}} $A_{\infty}$-modules $M^{\bullet}$ which means $m_2(1\otimes x)=x$ for all $x \in M^{\bullet}$ and 
$m_n(a_2\otimes \cdots \otimes a_n\otimes x)=0$ for $n>2$ if some $a_i=1$.

Morphisms between $A_{\infty}$-modules $M^{\bullet}$ and $N^{\bullet}$ are defined analogously to Definition \ref{Amor}. A morphism $f: M^{\bullet}\rightarrow N^{\bullet}$ is a collection of 
homogeneous $k$-linear maps
$$
f_n: A^{\bullet \otimes n-1}\otimes M^{\bullet} \longrightarrow N^{\bullet} \y \y \y \y \y n\geq 1,
$$ 
where $f_n$ has degree $1-n$ which are required to satisfy the identity (cf. \cite[Eqn.~(4.2), p. 15]{Kel01}):
$$
\sum (-1)^{r+st} f_{r+1+t}(1^{\otimes r}\otimes m_s\otimes 1^{\otimes t})=\sum (-1)^{u(v+1)} \nu_{u+1}(1^{\otimes u}\otimes f_v)
$$
with conventions in the $t=0$ case as above. On the right we are summing over decompositions $n=u+v$ with $v\geq 1$ and $u$ non-negative. Again we will only consider
{\it{strictly unital}} $f$ which means $f_n(a_2\otimes \cdots \otimes a_n\otimes x)=0$ for $n>1$ if some $a_i=1$.

There is an obvious identity morphism
$\text{Id}: M^{\bullet}\rightarrow M^{\bullet}$ and a natural way to compose morphisms, which results in the category $C_{\infty}(A^{\bullet})$ of strictly unital left $A_{\infty}$-modules over $A^{\bullet}$, with strictly unital morphisms.

There is a notion of a morphism $f: M^{\bullet}\rightarrow N^{\bullet}$ being null-homotopic which we will not recall in detail here, see \cite[p.~16]{Kel01} for the explicit formula expressing $f_n$ in terms of a homotopy $h_n: A^{\bullet \otimes n-1}\otimes M^{\bullet} \rightarrow N^{\bullet}$ (here $h_n$ has degree $-n$). In the homotopy category 
$K_{\infty}(A^{\bullet})$ morphisms are taken modulo homotopy. To arrive at the derived category we should invert all quasi-isomorphisms in $K_{\infty}(A^{\bullet})$, but another fundamental result of Kadeishvili says this is unnecessary:

\begin{thm}\cite[Thm.~4.2]{Kel01}
Let $f: M^{\bullet}\rightarrow N^{\bullet}$ be a quasi-isomorphism of $A_{\infty}$-modules. Then $f$ has a homotopy inverse morphism $g: N^{\bullet} \rightarrow M^{\bullet}$
of $A_{\infty}$-modules. ($f\circ g \sim \text{Id}_{N^{\bullet}}$ and $g\circ f \sim \text{Id}_{M^{\bullet}}$.)
\end{thm}

Thus $D_{\infty}(A^{\bullet})=K_{\infty}(A^{\bullet})$ is the derived category of $A_{\infty}$-modules over $A^{\bullet}$. It naturally becomes a triangulated category as discussed in \cite[Prop.~5.2]{Kel01}. For instance the translation functor is given by the shifts $T(M^{\bullet})^i=M^{i+1}$ with higher multiplications $m_n^{T(M^{\bullet})}=(-1)^n m_n^{M^{\bullet}}$. 

When $A^{\bullet}$ is a DGA any DG-module is of course an $A_{\infty}$-module (with $\nu_n=0$ for $n\geq 3$); but not conversely. Given a complex $(M^{\bullet},\nu_1)$ there could be morphisms $A^{\bullet}\rightarrow \End_k^{\bullet}(M^{\bullet})$ of $A_{\infty}$-algebras which do not arise from a DG-module (cf. the last paragraph of \cite[p.~15]{Kel01}). However, on the level of derived categories we have an equivalence by a result from Lef\`{e}vre-Hasegawa's 2003 Paris 7 Ph.D. thesis.

\begin{lem}\label{dga}\cite[Lem.~4.1.3.8]{LH03}
Let $A^{\bullet}$ be a DGA (unital by our conventions). The inclusion of the category of (unital) DG-modules $C(A^{\bullet})\rightarrow C_{\infty}(A^{\bullet})$ induces an equivalence of triangulated categories
$$
D(A^{\bullet})\overset{\sim}{\longrightarrow} D_{\infty}(A^{\bullet}).
$$
\end{lem}

Another key result from Lef\`{e}vre-Hasegawa's treatise which we will need concerns restriction of modules along a morphism of $A_{\infty}$-algebras
$f: A^{\bullet} \rightarrow B^{\bullet}$. Say $M^{\bullet}$ is an $A_{\infty}$-module over $B^{\bullet}$ with higher structure maps $(\nu_n^B)$. On the same vector space 
$M^{\bullet}$ we define an $A_{\infty}$-module structure over $A^{\bullet}$ by letting 
$$
\nu_n^A: A^{\bullet \otimes n-1}\otimes M^{\bullet}\longrightarrow M^{\bullet} \y \y \y \y \y \nu_n^A=\sum (-1)^{\sigma}\nu_{r+1}^B(f_{i_1}\otimes \cdots\otimes f_{i_r}\otimes \text{Id})
$$ 
where the sum extends over $1\leq r\leq n-1$ and decompositions $i_1+\cdots+i_r=n-1$. The sign $(-1)^{\sigma}$ is as in definition \ref{Amor}, cf. \cite[Sect.~6.2]{Kel01}. The ensuing module is denoted by $f^*M^{\bullet}$ .

\begin{thm}\label{res}\cite[Thm.~4.1.2.4]{LH03}
If $f: A^{\bullet} \rightarrow B^{\bullet}$ is a quasi-isomorphism of $A_{\infty}$-algebras, the restriction functor $f^*:C_{\infty}(B^{\bullet})\rightarrow C_{\infty}(A^{\bullet})$
induces an equivalence of triangulated categories
$$
f^*:D_{\infty}(B^{\bullet})\overset{\sim}{\longrightarrow} D_{\infty}(A^{\bullet}).
$$
\end{thm}

In our applications $A^{\bullet}$ will be a DGA and we endow $h^*(A^{\bullet})$ with an $A_{\infty}$-structure as in Kadeishvili's theorem \ref{kade}. By lemma \ref{dga}
and theorem \ref{res} a choice of quasi-isomorphism $f: h^*(A^{\bullet})\rightarrow A^{\bullet}$ gives an equivalence $f^*:D(A^{\bullet})\overset{\sim}{\longrightarrow}
D_{\infty}(h^*(A^{\bullet}))$ of triangulated categories.

\section{Putting the pieces together}

\subsection{Proof of theorem \ref{one}}

We return to the setup of section \ref{yon}. Thus $k$ is a field of characteristic $p>0$ and $G$ is a compact $p$-adic Lie group which is assumed to be torsionfree and pro-$p$. 
Recall that $\Omega=k[\![G]\!]$ and $D(\Omega)$ is the derived category of pseudocompact $\Omega$-modules. We fix a projective resolution $P^{\bullet}\rightarrow k$ and consider the DGA $\mathcal{H}^{\bullet} \overset{\sim}{\longrightarrow} \End_{\Omega}^{\bullet}(P^{\bullet})$ whose cohomology is the Yoneda algebra $h^*(\mathcal{H}^{\bullet}) \overset{\sim}{\longrightarrow} \Ext_{\Omega}^*(k,k)=E(\Omega)$. Furthermore, by Schneider's theorem \ref{main} and Pontryagin duality (corollary \ref{pd}) there is an equivalence
$$
D(\Omega)^{\op} \overset{\sim}{\longrightarrow} D(\End_{\Omega}^{\bullet}(P^{\bullet}))
$$
which on objects takes a $K$-projective complex $X^{\bullet}$ of pseudocompact $\Omega$-modules to the left DG module $\Hom_{\Omega}^{\bullet}(X^{\bullet},P^{\bullet})$. Now, by Kadeishvili's theorem \ref{kade} the cohomology $h^*(\mathcal{H}^{\bullet})\overset{\sim}{\longrightarrow} E(\Omega)$ acquires a natural $A_{\infty}$-algebra structure 
(unique up to non-unique isomorphism) and we may choose a quasi-isomorphism of $A_{\infty}$-algebras $f: h^*(\mathcal{H}^{\bullet}) \rightarrow \mathcal{H}^{\bullet}$ which by lemma \ref{dga} and theorem \ref{res} yields an equivalence
$$
f^*: D(\mathcal{H}^{\bullet}) \overset{\sim}{\longrightarrow} D_{\infty}(E(\Omega)).
$$
Altogether this gives an equivalence $D(\Omega)^{\op} \overset{\sim}{\longrightarrow} D_{\infty}(E(\Omega))$ taking $X^{\bullet}$ to $f^*\Hom_{\Omega}^{\bullet}(X^{\bullet},P^{\bullet})$.

\begin{rem}
There is a natural way to define the {\it{opposite}} $A^{\bullet \op}$ of an $A_{\infty}$-algebra $A^{\bullet}$, cf. \cite[Sect.~3.5]{LPWZa}. We take the same graded vector space 
but the opposite higher multiplication maps $m_n^{\op}: (A^{\bullet \op})^{\otimes n}\rightarrow A^{\bullet \op}$ are defined as $m_n^{\op}=(-1)^{\binom{n}{2}+1}\cdot m_n\circ (\text{reverse})$ where "reverse" is the map reversing the factors in the tensor product. For example, when applied to homogeneous elements $a_i$ we have
$$
m_n^{\op}(a_1\otimes \cdots \otimes a_n)=(-1)^{\binom{n}{2}+1+\sum_{i<j} \deg(a_i)\deg(a_j)}\cdot m_n(a_n \otimes \cdots \otimes a_1).
$$
By \cite[Lem.~3.3]{LPWZa} this gives an $A_{\infty}$-algebra $A^{\bullet \op}$ and one can easily verify that there is an equivalence
$$
D_{\infty}(A^{\bullet \op})  \overset{\sim}{\longrightarrow} D_{\infty}(A^{\bullet})^{\op}.
$$
(The opposite of a triangulated category was defined in section \ref{not}.)
\end{rem}

We define the {\it{Koszul}} dual of $\Omega$ as follows.

\begin{defn}
$\Omega^!=E(\Omega)^{\op}=\Ext_{\Omega}^*(k,k)^{\op}$ (with its $A_{\infty}$-structure).
\end{defn}

Then we conclude that $D(\Omega) \overset{\sim}{\longrightarrow} D_{\infty}(\Omega^!)$ by taking opposites on both sides above.

\subsection{Proof of theorem \ref{two}}\label{prooftwo}

Now let $G$ be a uniform pro-$p$ group. As explained after corollary \ref{cohom} it carries a natural $p$-valuation $\omega$ such that $(G,\omega)$ becomes equi-$p$-valued with $t=1$ and saturated. Thus
$$
\Omega^!=\Ext_{\Omega}^*(k,k)^{\op} \overset{\ref{gomega}}{=\joinrel=\joinrel=} \Ext_{G}^*(k,k)=\joinrel=\joinrel=H^*(G,k) \overset{\ref{cohom}}{=\joinrel=\joinrel=}
\bigwedge \frak{g}^{\vee}
$$
as graded $k$-algebras from which the first half of theorem \ref{two} follows.

It remains to show $\mu_n=0$ for $n >2$ when $G$ is abelian. Note that any abelian $p$-valuable group $G$ of rank $d$ is topologically isomorphic to $\Z_p^d$ by choosing a basis
$(g_1,\ldots,g_d)$. In particular we can always endow $G$ with a $p$-valuation $\omega$ in such a way that $(G,\omega)$ becomes equi-$p$-valued with $t=1$ and saturated; just let $\omega(g)=1+\min_{i=1,\ldots,d} v(x_i)$ when $g=g_1^{x_1}\cdots g_d^{x_d}$. As always we are assuming $p>2$ here. Correspondingly we have an isomorphism of topological filtered $k$-algebras
$$
k[\![X_1,\ldots,X_d]\!] \overset{\sim}{\longrightarrow} \Omega \y \y \y \y \y \sum_{\alpha} c_{\alpha}X^{\alpha} \mapsto \sum_{\alpha} c_{\alpha}{\bf{b}}^{\alpha}.
$$
We prefer to free $\Omega$ from coordinates and think of it as a completed symmetric algebra as follows. Fix a $d$-dimensional vector space $V$ over $k$ and consider the $I$-adic completion $\widehat{S(V)}=\varprojlim S(V)/I^i$ where $I=\ker\big(S(V)\longrightarrow k\big)$ is the augmentation ideal. Fixing a basis $(e_1,\ldots,e_d)$ for $V$ the map which takes 
${\bf{e}}^{\alpha}=e_1^{\alpha_1}\cdots e_d^{\alpha_d}$ to ${\bf{b}}^{\alpha}$ extends to an isomorphism $\widehat{S(V)}\overset{\sim}{\longrightarrow} \Omega$ of topological $k$-algebras. Note that $\widehat{S(V)}$ is flat over $S(V)$ so $I$-adic completion is exact on the category of finitely generated $S(V)$-modules. In particular we obtain a projective resolution $P^{\bullet}\rightarrow k$ in $\Mod(\Omega)$ by taking the $I$-adic completion of the Koszul resolution $K^{\bullet}$. That is, 
$$
K^{\bullet} \y \cdots \longrightarrow S(V) \otimes \bigwedge^2 V \longrightarrow S(V) \otimes \bigwedge^1 V \longrightarrow S(V)\longrightarrow k \longrightarrow 0,
$$
and
$$
P^{\bullet} \y \cdots \longrightarrow \widehat{S(V)} \otimes \bigwedge^2 V \longrightarrow \widehat{S(V)} \otimes \bigwedge^1 V \longrightarrow \widehat{S(V)}\longrightarrow k \longrightarrow 0.
$$
Recall that we use cohomological indexing which means $K^{-i}=S(V) \otimes \bigwedge^i V$ and similarly for $P^{\bullet}$. We use this particular completed Koszul resolution $P^{\bullet}$ in the definition of $\HH^{\bullet}$. That is $\HH^{\bullet}=\End_{\Omega}^{\bullet}(P^{\bullet})$, which has cohomology
$h^*(\HH^{\bullet})=E(\Omega)$. By theorem \ref{kade} the $A_{\infty}$-structure on $E(\Omega)$ is characterized up to $A_{\infty}$-isomorphism by the fact that there is a
quasi-isomorphism
$$
f: h^*(\HH^{\bullet})=\joinrel=\joinrel=E(\Omega) \longrightarrow \HH^{\bullet}.
$$
Analogously we have the differential graded algebra $\End_{S(V)}^{\bullet}(K^{\bullet})$ obtained from the (actual) Koszul resolution, which has cohomology $E(S(V))$. Again by theorem \ref{kade} there is an $A_{\infty}$-structure on $E(S(V))$ admitting a quasi-isomorphism
$$
g: h^*\big(\End_{S(V)}^{\bullet}(K^{\bullet})\big)=\joinrel=\joinrel=E(S(V))\longrightarrow \End_{S(V)}^{\bullet}(K^{\bullet}).
$$
By a well-known result a finitely generated graded algebra $A$ is Koszul if and only if the $A_{\infty}$-algebra $\Ext_A^*(k,k)$ is formal, cf. \cite[Sect.~2.2]{Kel02} and \cite[Thm.~6.3.3]{Wit18}. We deduce that $E(S(V))$ has the trivial $A_{\infty}$-structure since $S(V)$ is Koszul. We are therefore done once we check that the natural map
$$
\End_{S(V)}^{\bullet}(K^{\bullet}) \longrightarrow \End_{\widehat{S(V)}}^{\bullet}(P^{\bullet})
$$
which takes a homogeneous $a=(a_q)_{q\in \Z}$ to the corresponding sequence $\hat{a}=(\hat{a}_q)_{q\in \Z}$ of maps between completions, is a quasi-isomorphism. 
This follows from the standard calculation (and the trivial observation that the differentials of $\Hom_{S(V)}(S(V) \otimes \bigwedge^{\bullet} V,k)$ all vanish since $V \hookrightarrow S(V)$ augments to zero)
\begin{equation}\label{com}
\begin{split}
h^i\big(\End_{S(V)}^{\bullet}(K^{\bullet})\big)&=\Ext_{S(V)}^i(k,k)\\
  &=h^i\big(\Hom_{S(V)}(K^{\bullet},k)\big)\\
  &=h^i\big(\Hom_{S(V)}(S(V) \otimes \bigwedge^{\bullet} V,k) \big)\\
  &=h^i\big(\Hom_{k}(\bigwedge^{\bullet} V,k) \big)\\
  &=\bigwedge^i V^{\vee}
\end{split}
\end{equation}
together with the completely analogous computation of the cohomology of $\End_{\widehat{S(V)}}^{\bullet}(P^{\bullet})$ which in particular shows that -- as graded $k$-algebras -- we have isomorphisms
$$
E(S(V)) \overset{\sim}{\longrightarrow} E(\widehat{S(V)}) \overset{\sim}{\longrightarrow}\bigwedge V^{\vee}.
$$ 
Since they are all minimal models of $\HH^{\bullet}$ they are $A_{\infty}$-isomorphic. As observed above $E(S(V))$ has a trivial $A_{\infty}$-structure, and consequently so does 
$E(\Omega)$ and its opposite $\Omega^!=\bigwedge \frak{g}^{\vee}$. This finishes the proof. 

\begin{rem}\label{d=1}
The case $d=1$ of theorem \ref{two} is due to Schneider. In \cite[Sect.~5.2(4)]{Sch15} it is shown that there is an equivalence
$D(\Z_p)\overset{\sim}{\longrightarrow} D(k[\varepsilon])$ where $k[\varepsilon]=k \oplus k\varepsilon$ is the algebra of dual numbers ($\varepsilon^2=0$) thought of as a DGA concentrated in degrees $0$ and $1$, with zero differential. Note that DG-modules over $k[\varepsilon]$ are the same as graded $k$-vector spaces 
$M^{\bullet}=\oplus_{i\in \Z} M^i$ with two anti-commuting differentials $d$ and $\varepsilon$ of degree one, i.e. $d\varepsilon+\varepsilon d=0$.
Obviously $k[\varepsilon]\simeq \bigwedge k$ so Schneider's result also follows from theorem \ref{two} by observing that
$\bigwedge k$ is intrinsically formal (every minimal $A_{\infty}$-structure on it is trivial); which can be deduced from a computation of its Hochschild cohomology
$HH^*(\bigwedge k,\bigwedge k)$ as in \cite[Ex.~1.1.16]{Wit18} together with Kadeishvili's criterion \cite[Thm.~1]{Kad09} which says that every $A_{\infty}$-structure on a graded algebra $R$ is degenerate if $HH^{n,2-n}(R,R)=0$ for all $n>2$. 

Note that when $G$ is abelian the DGA $\bigwedge \frak{g}^{\vee}$ is graded commutative and therefore isomorphic to its opposite $\big(\bigwedge \frak{g}^{\vee}\big)^{\op}$. (We refer the reader to \cite[Sect.~3.4]{LPWZa} for its definition.) As a result theorem \ref{two} may also be thought of as an equivalence $D(G) \overset{\sim}{\longrightarrow} D(\bigwedge \frak{g}^{\vee})$ in the case where $G\simeq \Z_p^d$. Once we fix a basis $(\xi_i)$ for $\frak{g}$, a DG-module over $\bigwedge \frak{g}^{\vee}$ can be identified with a graded 
$k$-vector space $M^{\bullet}=\oplus_{i\in \Z} M^i$ with $d+1$ anti-commuting degree one differentials $d, \varepsilon_1,\ldots, \varepsilon_d$ -- we apologize for the double meaning of $d$ which refers to both $d_{M^{\bullet}}$ and $\dim G$. (Here $\varepsilon_i$ is the action of 
the dual basis vector $\xi_i^{\vee}\in \frak{g}^{\vee}$ on $M^{\bullet}$.) 
\end{rem}

\section{Open ends and unanswered questions}

We finish with a few questions which have not been addressed in this paper.

\begin{itemize}
\item[(a)] By Theorem \ref{one} one can recover $\Omega=k[\![G]\!]$ up to derived equivalence from the $A_{\infty}$-algebra $\Omega^!=\Ext_{\Omega}^*(k,k)^{\op}$.
Does $\Omega^!$ determine $\Omega$ up to isomorphism? 

\smallskip

\item[(b)] Is there a converse to Theorem \ref{two} in the sense that $G$ must be abelian if the $A_{\infty}$-structure on $\bigwedge \frak{g}^{\vee}$ is trivial? 

\smallskip

\item[(c)] Suppose $H \subset G$ is an open subgroup. Then $\Omega(G)$ is finite free over the subalgebra $\Omega(H)$ and we have the restriction map
$\Mod \Omega(G) \rightarrow \Mod \Omega(H)$ which induces a map $D(\Omega(G))\rightarrow D(\Omega(H))$. Is there a morphism of $A_{\infty}$-algebras
$\bigwedge \frak{g}^{\vee}\longrightarrow \bigwedge \frak{h}^{\vee}$ inducing the corresponding map on $D_{\infty}$ via "extension of scalars" along this map?
\end{itemize} 

Several experts in the field have independently confirmed that the answer to all three questions is yes, but it would take us too far afield to try to reproduce their arguments here. 



\subsection*{Acknowledgments} This work emerged from several inspiring conversations with Peter Schneider during my visit to M\"{u}nster in July 2015, and after. He led me to think about Koszul duality and $A_{\infty}$-algebras, and I am truly grateful for his guidance. The inspiration drawn from \cite{Sch15} should be clear to the reader. After sending out an initial draft of the article in the Summer of 2018, I received tremendously helpful remarks and suggestions from Carl Wang-Erickson, Michael Harris, Bernhard Keller, Karol Koziol, 
Leonid Positselski, Niccolo' Ronchetti, and Peter Schneider. Special thanks go to Positselski for making extensive comments and sharing his expertise in the area, as well as Harris for pointing to potential applications. 



\bigskip

\noindent {\it{E-mail address}}: {\texttt{csorensen@ucsd.edu}}

\noindent {\sc{Department of Mathematics, UCSD, La Jolla, CA, USA.}}


\begin{thebibliography}{xxxxxxx}

\bibitem[BGG78]{BGG78} I. N. Bernstein, I. M. Gelfand, and S. I. Gelfand, {\it{Algebraic vector bundles on $\Bbb{P}^n$ and problems of linear algebra}}. (Russian) Funktsional. Anal. i Prilozhen. 12 (1978), no. 3, 66--67.

\bibitem[BGS96]{BGS96} A. Beilinson, V. Ginzburg, and W. Soergel, {\it{Koszul duality patterns in representation theory}}. J. Amer. Math. Soc. 9 (1996), no. 2, 473--527.

\bibitem[BL94]{BL94} J. Bernstein and V. Lunts, {\it{Equivariant sheaves and functors}}. Lecture Notes in Mathematics, 1578. Springer-Verlag, Berlin, 1994.

\bibitem[Bru66]{Bru66} A. Brumer, {\it{Pseudocompact algebras, profinite groups and class formations}}. J. Algebra 4 (1966) 442--470.

\bibitem[CE99]{CE99} H. Cartan and S. Eilenberg, {\it{Homological algebra}}. With an appendix by David A. Buchsbaum. Reprint of the 1956 original. Princeton Landmarks in Mathematics. Princeton University Press, Princeton, NJ, 1999.

\bibitem[DdSMS]{DdSMS} J. D. Dixon, M. P. F. du Sautoy, A. Mann, and D. Segal, {\it{Analytic pro-p groups}}. Second edition. Cambridge Studies in Advanced Mathematics, 61. Cambridge University Press, Cambridge, 1999.

\bibitem[Dem86]{Dem86} M. Demazure, {\it{Lectures on p-divisible groups}}. Reprint of the 1972 original. Lecture Notes in Mathematics, 302. Springer-Verlag, Berlin, 1986.

\bibitem[Eme10]{Eme10} M. Emerton, {\it{Ordinary parts of admissible representations of p-adic reductive groups I. Definition and first properties}}. Ast\'{e}risque No. 331 (2010), 
355--402.

\bibitem[Flo06]{Flo06} G. Fl\o{}ystad, {\it{Koszul duality and equivalences of categories}}. Trans. Amer. Math. Soc. 358 (2006), no. 6, 2373--2398.

\bibitem[Gab62]{Gab62} P. Gabriel, {\it{Des cat\'{e}gories ab\'{e}liennes}}. Bull. Soc. Math. France 90 (1962) 323--448.


\bibitem[Har66]{Har66} R. Hartshorne, {\it{Residues and duality}}. Lecture notes of a seminar on the work of A. Grothendieck, given at Harvard 1963/64. With an appendix by P. Deligne. Lecture Notes in Mathematics, No. 20 Springer-Verlag, Berlin-New York 1966.

\bibitem[HKN11]{HKN11} A. Huber, G. Kings, and N. Naumann, {\it{Some complements to the Lazard isomorphism}}. Compos. Math. 147 (2011), no. 1, 235--262.

\bibitem[Kad83]{Kad83} T. Kadeishvili, {\it{The algebraic structure in the homology of an $A(\infty)$-algebra}}. Soobshch. Akad. Nauk Gruzin. SSR 108 (1982), no. 2, 249--252 (1983).

\bibitem[Kad86]{Kad86} T. Kadeishvili, {\it{Twisted tensor products for the category of $A(\infty)$-algebras and $A(\infty)$-modules}}. Trudy Tbiliss. Mat. Inst. Razmadze Akad. Nauk Gruzin. SSR 83 (1986), 26--45.

\bibitem[Kad88]{Kad88} T. Kadeishvili, {\it{The structure of the $A(\infty)$-algebra, and the Hochschild and Harrison cohomologies}}. Trudy Tbiliss. Mat. Inst. Razmadze Akad. Nauk Gruzin. SSR 91 (1988), 19--27.

\bibitem[Kad09]{Kad09} T. Kadeishvili, {\it{Cohomology $C_{\infty}$-algebra and rational homotopy type}}. Algebraic topology--old and new, 225--240, Banach Center Publ., 85, Polish Acad. Sci. Inst. Math., Warsaw, 2009.

\bibitem[Kel94]{Kel94} B. Keller, {\it{Deriving DG categories}}. Ann. Sci. \'{E}cole Norm. Sup. (4) 27 (1994), no. 1, 63--102.

\bibitem[Kel01]{Kel01} B. Keller, {\it{Introduction to A-infinity algebras and modules}}. Homology Homotopy Appl. 3 (2001), no. 1, 1--35.

\bibitem[Kel02]{Kel02} B. Keller, {\it{A-infinity algebras in representation theory}}. Representations of algebra. Vol. I, II, 74--86, Beijing Norm. Univ. Press, Beijing, 2002.

\bibitem[Kel06]{Kel06} B. Keller, {\it{A-infinity algebras, modules and functor categories}}. Trends in representation theory of algebras and related topics, 67--93, Contemp. Math., 406, Amer. Math. Soc., Providence, RI, 2006.

\bibitem[Kel07]{Kel07} B. Keller, {\it{Derived categories and tilting}}. Handbook of tilting theory, 49--104, London Math. Soc. Lecture Note Ser., 332, Cambridge Univ. Press, Cambridge, 2007.

\bibitem[Koh17]{Koh17} J. Kohlhaase, {\it{Smooth duality in natural characteristic}}. Adv. Math. 317 (2017), 1--49. 

\bibitem[Kon95]{Kon95} M. Kontsevich, {\it{Homological algebra of mirror symmetry}}. Proceedings of the International Congress of Mathematicians, Vol. 1, 2 (ZŸrich, 1994), 120--139, Birkh\"{a}user, Basel, 1995.

\bibitem[Kra07]{Kra07} H. Krause, {\it{Derived categories, resolutions, and Brown representability}}. Interactions between homotopy theory and algebra, 101--139, Contemp. Math., 436, Amer. Math. Soc., Providence, RI, 2007.

\bibitem[Laz65]{Laz65} M. Lazard, {\it{Groupes analytiques p-adiques}}. Inst. Hautes \'{E}tudes Sci. Publ. Math. No. 26 (1965) 389--603.

\bibitem[LH03]{LH03} K. Lef\`{e}vre-Hasegawa, {\it{Sur les $A_{\infty}$-cat\'{e}gories}}. Th\`{e}se de Doctorat, Universite Paris 7, November 2003. Available online at \texttt{arXiv:math/0310337}.

\bibitem[LPWZa]{LPWZa} D.-M. Lu, J.H. Palmieri, Q.-S. Wu, and J. J. Zhang, {\it{Koszul equivalences in $A_{\infty}$-algebras}}. 
New York J. Math. 14 (2008), 325--378. 

\bibitem[LPWZb]{LPWZ} D.-M. Lu, J. H. Palmieri, Q.-S. Wu, and J. J. Zhang, {\it{A-infinity structure on Ext-algebras}}. J. Pure Appl. Algebra 213 (2009), no. 11, 2017--2037.

\bibitem[Mer99]{Mer99} S. Merkulov, {\it{Strong homotopy algebras of a K\"{a}hler manifold}}. Internat. Math. Res. Notices 1999, no. 3, 153--164.

\bibitem[Nee01]{Nee01} A. Neeman, {\it{Triangulated categories}}. Annals of Mathematics Studies, 148. Princeton University Press, Princeton, NJ, 2001.

\bibitem[OS18]{OS18} R. Ollivier and P. Schneider, {\it{A canonical torsion theory for pro-p Iwahori-Hecke modules}}. Adv. Math. 327 (2018), 52--127.

\bibitem[Pos93]{Pos93} L. E. Positselski, {\it{Nonhomogeneous quadratic duality and curvature}}. 
Funct. Anal. Appl. 27 (1993), no. 3, 197--204.

\bibitem[Pos11]{Pos11} L. Positselski, 
{\it{Two kinds of derived categories, Koszul duality, and comodule-contramodule correspondence}}. 
Mem. Amer. Math. Soc. 212 (2011), no. 996, vi+133 pp.

\bibitem[Sch95]{Sch95} W. H. Schikhof, {\it{A perfect duality between p-adic Banach spaces and compactoids}}. 
Indag. Math. (N.S.) 6 (1995), no. 3, 325--339. 

\bibitem[Sch11]{Sch11} P. Schneider, {\it{p-adic Lie groups}}. Grundlehren der Mathematischen Wissenschaften [Fundamental Principles of Mathematical Sciences], 344. Springer, Heidelberg, 2011.

\bibitem[Sch15]{Sch15} P. Schneider, {\it{Smooth representations and Hecke modules in characteristic p}}. Pacific J. Math. 279 (2015), no. 1-2, 447--464.

\bibitem[Ser65]{Ser65} J.-P. Serre, {\it{Sur la dimension cohomologique des groupes profinis}}. Topology 3 (1965) 413--420.

\bibitem[Spa88]{Spa88} N. Spaltenstein, 
{\it{Resolutions of unbounded complexes}}.
Compositio Math. 65 (1988), no. 2, 121--154. 

\bibitem[Sta61]{Sta61} J. Stasheff, {\it{Homotopy associativity of $H$-spaces}}. Thesis (Ph.D.)--Princeton University. 1961.

\bibitem[Sta18]{Sta18} The Stacks Project Authors, {\it{Stacks Project}}, \texttt{http://stacks.math.columbia.edu/}, retrieved May 2018.

\bibitem[ST02]{ST02} P. Schneider and J. Teitelbaum, {\it{Banach space representations and Iwasawa theory}}. 
Israel J. Math. 127 (2002), 359--380. 

\bibitem[Val14]{Val14} B. Vallette, {\it{Algebra + homotopy = operad}}. Symplectic, Poisson, and noncommutative geometry, 229--290, Math. Sci. Res. Inst. Publ., 62, Cambridge Univ. Press, New York, 2014.

\bibitem[Ven17]{Ven17} A. Venkatesh, {\it{Derived Hecke algebra and cohomology of arithmetic groups}}. Preprint, 2017. arXiv:1608.07234

\bibitem[WE18]{WE18} C. Wang-Erickson, {\it{Deformations of residually reducible Galois representations via $A_{\infty}$-algebra structure on Galois cohomology}}. Preprint, 2018. arXiv:1809.02484.

\bibitem[Wit18]{Wit18} S. Witherspoon, {\it{An Introduction to Hochschild Cohomology}}. Book draft, retrieved June 2018. Available online.

\end{thebibliography}
\end{document}